\documentclass[reqno]{amsart}

\usepackage[ascii]{inputenc}

\usepackage{amsmath}
\usepackage{amssymb}
\usepackage{mathtools}
\usepackage{todonotes}

\usepackage[linesnumbered,lined,boxruled,algosection]{algorithm2e}
\IncMargin{.5em}
\SetKw{KwBreak}{break}
\SetKwInput{Input}{Input}
\SetKwInput{Output}{Output}
\SetKwInput{Domain}{Domain}
\SetKwInput{Barrier}{Barrier}
\SetKwInput{ComplexityParameter}{Complexity parameter}
\SetKwFunction{MainStage}{MainStage}
\SetKwFunction{PreliminaryStage}{PreliminaryStage}
\SetKwComment{Comment}{$\triangleright$\ }{}
\SetCommentSty{scriptsize}

\usepackage{hyperref}
\usepackage{amsthm}
\usepackage[capitalize]{cleveref}
\theoremstyle{plain}
\newtheorem{thm}{Theorem}[section]\crefname{thm}{Theorem}{Theorems}
\newtheorem{prop}[thm]{Proposition}\crefname{prop}{Proposition}{Propositions}
\newtheorem{lem}[thm]{Lemma}\crefname{lem}{Lemma}{Lemmas}
\newtheorem{cor}[thm]{Corollary}\crefname{cor}{Corollary}{Corollaries}
\theoremstyle{definition}
\newtheorem{defn}[thm]{Definition}\crefname{defn}{Definition}{Definitions}
\newtheorem{problem}[thm]{Problem}\crefname{problem}{Problem}{Problems}
\crefname{remark}{Remark}{Remarks}
\newtheorem{example}[thm]{Example}\crefname{example}{Example}{Examples}
\crefformat{section}{#2\S#1#3}

\theoremstyle{plain}
\newtheorem*{ugpthmwc}{{\Cref{thm:algo wc}}}
\newtheorem*{ugpthmgeneral}{{\Cref{thm:algo general}}}

\newcommand{\C}{\mathbb{C}}
\renewcommand{\int}{\NOINTEGRALSHERE}

\newcommand{\R}{\mathbb{R}}
\newcommand{\Q}{\mathbb{Q}}
\newcommand{\Z}{\mathbb{Z}}

\newcommand{\T}{\mathrm{T}}

\newcommand{\poly}{\mathrm{poly}}
\renewcommand{\epsilon}{\varepsilon}
\newcommand{\sym}[1]{\operatorname{sym}({#1})}
\newcommand{\enclen}[1]{\langle {#1} \rangle}
\newcommand{\norm}[1]{\left\lVert{#1}\right\rVert}
\newcommand{\tnorm}[1]{\left|\mkern-1.5mu\left|\mkern-1.5mu\left|#1\right|\mkern-1.5mu\right|\mkern-1.5mu\right|}
\newcommand{\abs}[1]{\left\lvert{#1}\right\rvert}

\newcommand{\ip}[2]{\langle#1,#2\rangle}

\newcommand{\grad}{\nabla}
\renewcommand{\vec}[1]{#1}
\newcommand{\val}{\mathrm{val}}

\DeclareMathOperator{\spec}{spec}
\DeclareMathOperator{\ufc}{ufc}
\DeclareMathOperator{\GL}{GL}

\DeclareMathOperator{\diag}{diag}
\DeclareMathOperator{\KL}{KL}
\DeclareMathOperator{\capa}{cap}
\DeclareMathOperator{\interior}{int}
\DeclareMathOperator{\linspan}{span}
\DeclareMathOperator{\affspan}{aff}

\DeclareMathOperator{\conv}{conv}
\DeclareMathOperator{\Tr}{Tr}
\DeclareMathOperator{\relint}{relint}
\numberwithin{equation}{section}
\allowdisplaybreaks

\usepackage{enumitem}
\setlist[enumerate]{label=(\alph*),noitemsep}

\newif\ifdetailedbounds

\newcommand{\ugpthmgeneralcontent}{%
  There is an interior-point algorithm (\cref{algo:ipm general}) that, given as input an instance of the unconstrained GP problem (\cref{prob:approximate GP}) with shift~$\theta\in\conv\Omega$ and a lower bound $0 < \varphi_0 \leq \varphi$ on the facet gap, returns $\vec{x}_\delta \in \R^n$ such~that
  \begin{align*}
    F_\theta(\vec{x}_\delta) \leq F_\theta^* + \delta
  \end{align*}
  within
  \begin{align*}
  \ifdetailedbounds
  41 \sqrt{k} \log\left( 3600 \, k^2 n \frac N {\varphi_0} \frac1\delta \log^2\left( \frac{5k\beta}{\delta} \right) \right)
  =
  \fi
  O\left( \sqrt{k} \log\left( k n \frac N {\varphi_0} \frac1\delta \log\left( \frac{k\beta}{\delta} \right) \right) \right)
  \end{align*}
  iterations.
  The starting point is determined explicitly by the input, and every iteration is a Newton step with respect to a known function that depends on $\varphi_0$.
}

\newcommand{\ugpthmwccontent}{%
  There exists an interior-point algorithm (\cref{algo:ipm well-conditioned}) that, given as input a well-conditioned instance of the unconstrained GP problem with shift (\cref{prob:approximate GP}),
  returns~$\vec{x}_{\delta} \in \R^n$ such that
  \begin{align*}
    F_\theta(\vec{x}_{\delta}) \leq F_\theta^* + \delta
  \end{align*}
  within
  \begin{align*}
  \ifdetailedbounds
  36 \sqrt k \log\left( 1440 k^2 \frac{R_\theta}{r_\theta} \frac1\delta \log^2(5k\beta) \right)
  =
  \fi
  O\left(\sqrt k \log\left( k \frac{R_\theta}{r_\theta} \frac1\delta \log(k\beta) \right) \right)
  \end{align*}
  iterations, where $r_\theta$ and $R_\theta$ are the geometric quantities from~\cref{def:condition numbers} and $\beta$~is defined in \cref{eq:beta}.
  The starting point of the algorithm is determined explicitly by the input, and every iteration is a Newton step for a known function.
}

\title[Interior-point methods for unconstrained geometric programming]{Interior-point methods for unconstrained geometric programming and scaling problems}
\date{}

\author{Peter B\"urgisser}
\address{Institut f\"ur Mathematik, Technische Universit\"at Berlin}
\email{pbuerg@math.tu-berlin.de}
\author{Yinan Li}
\address{Centrum Wiskunde \& Informatica (CWI) and QuSoft}
\email{liyinan9252@gmail.com}
\author{Harold Nieuwboer}
\address{Korteweg--de Vries Institute for Mathematics and QuSoft, University of Amsterdam}
\email{h.a.nieuwboer@uva.nl}
\author{Michael Walter}
\address{Korteweg--de Vries Institute for Mathematics, Institute for Theoretical Physics, Institute for Language, Logic, and Computation, and QuSoft, University of Amsterdam}
\email{m.walter@uva.nl}

\begin{document}

\begin{abstract}
  We provide a condition-based analysis of two interior-point methods for unconstrained geometric programs, a class of convex programs that arise naturally in applications including matrix scaling, matrix balancing, and entropy maximization.
  Our condition numbers are natural geometric quantities associated with the Newton polytope of the geometric program, and lead to diameter bounds on approximate minimizers.
  We also provide effective bounds on the condition numbers both in general and under combinatorial assumptions on the Newton polytope.
  In this way, we generalize the iteration complexity of recent interior-point methods for matrix scaling and matrix balancing.
  Recently, there has been much work on algorithms for certain optimization problems on Lie groups, known as capacity and scaling problems.
  For commutative groups, these problems reduce to unconstrained geometric programs, which serves as a particular source of motivation for our work.
\end{abstract}
\maketitle

\section{Introduction}
Geometric programming is an optimization paradigm that generalizes linear programming and has a wide range of applications~\cite{duffin1967geometric,Boyd2007}.
In this paper, we are concerned with \emph{unconstrained geometric programs}.
These are optimization problems of the form
\begin{equation}\label{eq:gp intro}
\begin{aligned}
  \text{minimize} \quad & f(\vec z) \\
  \text{subject to} \quad & \vec z \in \R^n, \, \vec z > \vec{0},
\end{aligned}
\end{equation}
where $f(\vec z)$ is a \emph{posynomial} in positive real variables $z_1,\dots,z_n$.
That is,
\begin{align}\label{eq:posy}
  f(\vec{z}) = \sum_{i=1}^k q_i \vec{z}^{\omega_i} = \sum_{i=1}^k q_i \prod_{j=1}^n z_j^{\omega_{i,j}},
\end{align}
where the coefficients $q_i$ are positive and the exponents $\omega_{i,j}$ are real numbers.
In a general geometric program (GP), one adds posynomial inequality and monomial equality constraints.
Although posynomials are non-convex in general, they are convex in~$\vec x$ after the change of variables~$\vec z = e^{\vec x}$.
As such, they are the simplest family of \emph{geodesically convex} programming problems.
It is well-known that standard convex programming techniques, like the ellipsoid method or interior-point methods, can solve GP in polynomial time, see for instance~\cite{nesterov-nemirovskii,kortanek1997infeasible,andersen1998computational,boyd-vandenberghe,gurvits2004combinatorial,Singh-Vishnoi-14,straszak-vishnoi-bitcomplexity,nemirovski2005cone,karimi2018primal}.
Furthermore, it has been observed that interior-point methods are also efficient and robust in practice (cf.~\cite{boyd-vandenberghe,Boyd2007}).
However, explicit iteration complexity bounds are not readily available in the current literature.

In this paper, we present a detailed analysis of two interior-point algorithms for unconstrained GP in terms of natural geometric condition numbers.
Our first algorithm applies to instances that (roughly speaking) have a well-conditioned Newton polytope, while our second algorithm has no such assumption but instead relies on a novel condition number for the GP.
We also provide effective bounds on these condition numbers both in general and under suitable combinatorial assumptions.
The latter assumptions apply in particular to matrix scaling and balancing and, as a consequence, we match the iteration complexity of the recent work~\cite{cmtv-matrix-scaling}.
In fact, the recent interest in generalized scaling problems is a particular source of motivation for this paper, and we expand on this connection in~\cref{subsubsec:opti}.

\subsection{Computational problems}
Let us discuss the computational problems associated with~\eqref{eq:gp intro} in more detail.
We first introduce some terminology and notation.
We write $\vec{q}=(q_1,\dots,q_k)\in \R^k_{++}$ for the vector of \emph{coefficients} and $\Omega=\{\omega_1,\dots,\omega_k\}\subseteq\R^n$ for the set of \emph{exponents} of the posynomial~\eqref{eq:posy}.
By the logarithmic change of variables $e^{x_j} = z_j$, $j \in [n]$, and $e^F = f$, we can write~\eqref{eq:gp intro}~as%
\footnote{Throughout this paper $\log$ is the natural logarithm (so with respect to base $e$).}
\begin{equation}\label{eq:gp}
  F^* = \inf_{\vec x\in\R^n} F(x), \quad F(x) = \log \sum_{i=1}^k q_ie^{\ip{\omega_i}{\vec{x}}},
\end{equation}
where $\ip{\cdot}{\cdot}$ denotes the standard inner product in $\R^n$.
The objective $F(x)$ is a convex function.
It will in general not attain its infimum $F^*$, as the domain is unbounded.
The problem of unconstrained geometric programming is to approximate the infimum to arbitrary precision:

\begin{problem}[Unconstrained GP]\label{prob:approximate GP zero}
Given as input the exponents $\omega_1,\dots,\omega_k\in\R^n$, coefficients $\vec{q} \in \R^k_{++}$, and a precision $\delta\!\in\!(0,1)$, find $\vec{x}_{\delta}\in\R^n$ such that \!$F(\vec{x}_{\delta})\!\leq\!F^* + \delta$.
\end{problem}

\noindent
Clearly, any solution to this problem provides a $(1+2\delta)$-multiplicative and a $(2\norm q_1 \delta)$-additive approximation to the value of the original geometric program~\eqref{eq:gp intro}.

\Cref{prob:approximate GP zero} depends crucially on the \emph{Newton polytope} of~$f$, which is defined as the convex hull of the exponents $\Omega=\{\omega_1,\dots,\omega_k\}$ (recall that all coefficients $q_j$ are assumed nonzero).
Two well-known important properties are:
\begin{enumerate}
\item the infimum exists (i.e., $F^* > -\infty$) if and only if $\vec0\in\conv\Omega$, and
\item the infimum is attained (i.e., $F^*=F(x)$ for some $x\in\R^n$, so in particular $F^*$ is finite) if and only if $\vec0\in\relint\conv\Omega$.
\end{enumerate}
Property~(a) characterizes when \cref{prob:approximate GP zero} has a solution.
It follows from the observation that $F$ is unbounded from below if and only if there exists some $x\in\R^n$ such that $\ip{\omega_i}x<0$ for all $i\in[k]$, which in turn is equivalent to $0\not\in\conv\Omega$ by Farkas' lemma.
Thus, deciding whether $F^*$ is finite or not can be done by testing membership in the Newton polytope (a linear programming problem).

Property~(b) may be interpreted as characterizing when the instance is \emph{well-conditioned}.
To get some intuition, one can verify that the \emph{gradient} $\grad F(x)$ is a convex combination of the exponents $\omega_i$, with positive coefficients.
Thus if $F$ has a minimum then $0$ is in relative interior of the Newton polytope.
We give a quantitative version of the reverse implication in \cref{section:condition and diameter}.

By convexity, \cref{prob:approximate GP zero} is directly related to the problem of minimizing the gradient $\grad F(x)$.
We refer to this as the associated \emph{scaling} problem, as it captures the well-known polynomial scaling and matrix scaling problems as important special cases (see~\cref{subsubsec:matrix scaling} for details).

\begin{problem}[Scaling problem]\label{prob:scaling zero}
Given as input the exponents $\omega_1,\dots,\omega_k\in\R^n$, coefficients $\vec{q} \in \R_{++}^k$, and a precision $\epsilon > 0$, find $\vec{x}_\epsilon \in \R^n$ such that $\norm{\grad F(\vec{x}_\epsilon)}_2\leq\epsilon$.
\end{problem}

\noindent
The scaling problem is feasible for all $\epsilon>0$ precisely when $F$ is bounded from below, that is, when the Newton polytope contains the origin.

In many applications, which we will elaborate on in~\cref{subsec:motivation}, one is naturally interested in varying a given geometric program by translating all exponents by some fixed vector~$\theta\in\R^n$.
For this, let
\begin{align}\label{eq:log gp shifted}
  F^*_\theta=\inf_{\vec{x}\in\R^n}F_\theta(\vec{x}), \quad
  F_\theta(\vec{x}) = \log \sum_{i=1}^k q_ie^{\ip{\omega_i - \theta}{\vec{x}}}
  = F(x) - \ip \theta x.
\end{align}
When we refer to the Newton polytope of such a shifted problem, we always mean the convex hull of the original exponents $\Omega=\{\omega_1,\dots,\omega_k\}$.

In the remainder of this paper we will focus on solving the following computational problems, which generalize \cref{prob:approximate GP zero,prob:scaling zero}:

\begin{problem}[Unconstrained GP with shift]\label{prob:approximate GP}
Given as input a shift $\theta\in\R^n$, exponents $\omega_1,\dots,\omega_k\in\R^n$, $q \in \R^k_{++}$, and a precision $\delta\in(0,1)$, find $x_\delta\in\R^n$ such that $F_\theta(x_\delta) \leq F_\theta^* + \delta$.
\end{problem}

\begin{problem}[Scaling problem with shift]\label{prob:scaling}
Given as input a shift $\theta\in\R^n$, exponents $\omega_1,\dots,\omega_k\in\R^n$, $q \in \R_{++}^k$, and a precision $\epsilon > 0$, find $x_\epsilon \in \R^n$ such that $\norm{\grad F_\theta(x_\epsilon)}_2=\norm{\grad F(\vec x_\epsilon) - \theta}_2\leq\epsilon$.
\end{problem}

\noindent
Throughout we will always assume that the shift is contained in the Newton polytope (i.e., $\theta\in\conv\Omega$).
This assumption is natural, since otherwise the two problems have no solution (the latter for small~$\epsilon$), as follows from properties~(a) and (b).

\subsection{Motivations and prior work}\label{subsec:motivation}
Before discussing our results we discuss three important motivations from machine learning and optimization, which reduce to unconstrained GPs and have been subject of intense recent research.
This will also shed more light on the connection between \cref{prob:approximate GP,prob:scaling}.

\subsubsection{Entropy maximization}
The Lagrange dual of the convex optimization problem~\eqref{eq:log gp shifted} is given by an entropy maximization problem.
More precisely,
\begin{equation}\label{eq:entropy maximization}
F^*_\theta = \inf_{x\in\R^n} F_\theta(x) = \sup \biggl\{ -D_{\KL}(\vec{p}\|\vec{q}) \;:\; \sum_{i=1}^k p_i\omega_i=\theta,~\sum_{i=1}^k p_i=1,~\vec{p} \geq 0 \biggr\},
\end{equation}
where $D_{\KL}(\vec{p}\Vert\vec{q}) =\sum_{i=1}^k p_i\log \frac{p_i}{q_i}$ denotes the \emph{Kullback--Leibler (KL) divergence} between a probability distribution $\vec{p}$ and the distribution~$\vec{q}$ (which need not be normalized).
Thus, the dual program~\eqref{eq:entropy maximization} is feasible when $\theta$ is in the Newton polytope, and the optimal solution is a probability distribution on~$\Omega \cong [k]$ with mean~$\theta$ that minimizes the KL divergence to the initial distribution~$\vec{q}$.
When $\vec q=(1,\dots,1)$ is the all-ones vector, $-D_{\KL}(\vec{p}\|\vec{q}) = \sum_{i=1}^k p_i \log \frac1{p_i}$ is the Shannon entropy of $p$.
In this case, \cref{eq:entropy maximization} amounts to the discrete \emph{entropy maximization} problem which naturally arises in machine learning and statistics, motivated by the maximum entropy principle~\cite{PhysRev.106.620,PhysRev.108.171}.

To solve the entropy maximization problem, \cite{Singh-Vishnoi-14,straszak-vishnoi-bitcomplexity} proposed ellipsoid methods for the equivalent geometric program~\eqref{eq:log gp shifted} that are tractable even when~$k$ is large.
They focused on the case that~$\Omega$ consists of integer vectors (which is already of substantial interest) and gave a priori diameter bounds as required for the ellipsoid method.
In~\cite{Singh-Vishnoi-14}, it was shown that
if $\theta$ is at a distance $\eta>0$ from the boundary of the Newton polytope then there is a minimizer~$\vec{x}^*$ of norm $\norm{\vec{x}^*}_2\leq\frac{\log k}{\eta}$.
In~\cite{straszak-vishnoi-bitcomplexity}, a diameter bound was obtained in terms of the \emph{unary facet complexity} of the Newton polytope:
if $\conv\Omega$ can be described by linear inequalities with integer coefficients in $\{-M,\dots,M\}$, then for \emph{any} $\theta\in\conv\Omega$ there is a $\delta$-approximate minimizer $\vec{x}_\delta$ to the problem~\eqref{eq:log gp shifted} with $\norm{\vec{x}_\delta}_2\leq R$ where $R = \poly(n,M,\log\frac{1}{\delta})$.
This bound is particularly useful if $\theta$ is very close to (or on) the boundary of the Newton polytope.
Below we show how to generalize these diameter bounds to the case that the Newton polytope is not integral.

More recently,~\cite{celis2019fair} discusses a second-order box-constrained Newton method for entropy maximization, based on results from~\cite{allen2017much}.
The number of required iterations for approximately minimizing the objective is polynomial in $R$, $\log \frac{1}{\delta}$ and a parameter known as the \textit{robustness parameter} of the objective.
Robustness of the objective gives local quadratic approximations, and the robustness parameter controls the diameter of the region in which these approximations hold.
A more general version of this second-order method is presented in~\cite{bfgoww-noncommutative-optimization}.

\subsubsection{Matrix scaling}\label{subsubsec:matrix scaling}
Let $M \in \R^{n \times n}$ be a matrix with non-negative entries, and suppose we are given non-negative vectors $\vec{r}, \vec{c} \in \R^n$ with $\norm{\vec r}_1 = \norm{\vec c}_1 = 1$.
Let $r(M)$ denote the vector of row sums of~$M$, and let $c(M)$ denote the vector of column sums.
Then the \emph{$(r,c)$-matrix scaling} problem asks us to rescale the rows and columns of~$M$ so that its row and column sums are given by $\vec r$ and $\vec c$, respectively.
That is, we wish to find positive diagonal matrices $L, R \in \R^n$ such that $r(LMR) = r$ and $c(LMR) = c$.
This problem is very well-studied and has a wide range of applications (see, e.g., \cite{allen2017much,cmtv-matrix-scaling} and references therein).
A strongly polynomial algorithm based on the Sinkhorn-Knopp algorithm is presented in~\cite{lsw98}, and recent works provide very fast algorithms~\cite{allen2017much,cmtv-matrix-scaling}.

The connection to unconstrained GP is as follows.
Consider the objective
\begin{align}\label{eq:matrix scaling posy}
  f(x_1, \dotsc, x_n, y_1, \dotsc, y_n) = \sum_{i,j=1}^n M_{ij} e^{x_i + y_j}
\end{align}
in $2n$ variables and let $\theta=(r,c)$.
Then,
\begin{align*}
  \nabla F_\theta(x,y) = \frac{\sum_{i,j} N_{ij} (e_i,e_j)}{\sum_{i,j} N_{ij}} - \theta = \frac{\bigl(r(N),c(N)\bigr)}{\norm N_1} - (r,c),
\end{align*}
where $N = LMR$ is the matrix rescaled by $L = \diag(e^{x_i})$ and $R = \diag(e^{y_j})$.
This shows that finding an approximate $(r,c)$-matrix scaling is precisely equivalent to \cref{prob:scaling} and can be achieved by solving a geometric program, as in \cref{prob:approximate GP}.
Indeed, state-of-the-art algorithms for matrix scaling (as well as the very similar \emph{matrix balancing} problem) are based on minimizing~$F_\theta$ or closely related functions, and our results can be understood as a generalization of the interior point method of~\cite{cmtv-matrix-scaling} to arbitrary unconstrained GPs.

\subsubsection{Optimization over linear group actions}\label{subsubsec:opti}
Let $G$ be a continuous group acting linearly on a complex vector space $V=\C^k$ by a representation $\pi\colon G\to\GL(V)$, with $\GL(V)$ the group of invertible linear operators.
The action partitions the vector space into orbits $\mathcal{O}_{\vec{v}} = \{ \pi(g)\vec{v} : g \in G\}$ for $\vec{v} \in V$.
A basic algorithmic problem is then to compute the smallest norm of any vector in an orbit, known as the \emph{capacity} $\capa(\vec{v})=\inf\{\norm{\vec{w}}_2:~\vec{w}\in \mathcal{O}_\vec{v}\}$ of~$\vec{v}$.

The capacity problem for suitable groups and actions captures natural important problems in computational complexity, algebra, analysis, quantum information, and statistics: see the recent developments~\cite{goww-operator-scaling,Garg2018,azglow-operator-scaling,brgisser_et_al:LIPIcs:2018:8351,10.1145/3188745.3188932,10.1145/3188745.3188794,8948667,bfgoww-noncommutative-optimization,franks2020rigorous,franks2020minimal,amendola2020invariant,derksen2020maximum}. 
In the special case when $G$ is commutative, estimating the capacity of a vector can be recast as an unconstrained GP~\cite{bfgoww-noncommutative-optimization}.
We briefly recall the connection.
Let $G=\T(n)$ be the group of invertible \emph{diagonal} complex $n\times n$ matrices.
Under very mild assumptions, $\pi$ has the following simple form: there exists an orthonormal basis $\vec{e}_1,\dots,\vec{e}_k$ of~$V$ and integer vectors~$\omega_1,\dots,\omega_k\in\Z^n$ such that $\pi(g) e_i = \prod_{j=1}^n g_j^{\omega_{i,j}} e_i$ for any $i\in[k]$ and $g=\diag(g_1,\dots,g_n)\in\T(n)$.
Accordingly, given an arbitrary vector $v=\sum_{i=1}^k v_i e_i$ we have
\begin{align*}
  \norm{\pi(g) v}_2^2
= \sum_{i=1}^k \abs{v_i}^2 \prod_{j=1}^n \abs{g_j}^{2\omega_{i,j}},
\end{align*}
By replacing $q_i = \abs{v_i}^2$ and $e^{x_j}=\abs{g_j}^2$, we see that computing $\log\capa(\vec{v})$ reduces to solving an unconstrained GP of the form~\eqref{eq:gp}.

To make this connection concrete, we briefly sketch how the matrix scaling problem discussed earlier arises as a special case.
Consider the group $G=\T(n) \times \T(n) \cong \T(2n)$ and the vector space~$V=\C^{n\times n}$ of complex $n\times n$-matrices, endowed with the Frobenius norm.
The action is given by $\pi(g,h)A = g A h$.
Then,
\begin{align*}
    \norm{\pi(g,h) A}_F^2
= \sum_{i,j} \abs{g_i A_{ij} h_j}^2.
\end{align*}
Therefore, substituting $\abs{g_i}^2 = e^{x_i}$, $\abs{h_j}^2 = e^{y_j}$, and $M_{ij} = \abs{A_{ij}}^2$ returns us precisely to the situation of \cref{eq:matrix scaling posy}.

Interestingly, most of the concepts used in the above discussions are not restricted to the case where $G$ is commutative.
We briefly recall some of the concepts and refer the readers to~\cite{bfgoww-noncommutative-optimization} for more details.
Consider a `nice' noncommutative group such as $G=\GL(n)$ acting on $V$ by $\pi\colon G\to\GL(V)$.
The function $F(g) = \log\norm{\pi(g)\vec{v}}_2$ is not convex but \emph{geodesically convex}.
Geodesic convexity can be thought of as a generalization of convexity in Euclidean spaces to arbitrary Riemannian manifolds.
Here this means that the function $F(\exp(tH) g)$ is convex in $t\in\R$ for any Hermitian $n\times n$ matrix~$H$ and $g\in\GL(n)$, where $\exp(\cdot)$ denotes the matrix exponential map.
We can minimize geodesically convex functions on Riemannian manifolds just as in the Euclidean case by minimizing the gradient~$\nabla F(g)$, which in the present context is the Hermitian $n\times n$ matrix defined by the property:
$\Tr(\nabla F(g) H)=\partial_{t=0}\log\norm{\pi(\exp(tH) g\vec{v}}_2$.
A surprising result is that the closure of the set $\{\spec(\nabla F(g)) : g \in G \}$, where $\spec(\cdot)$ denotes the vector of eigenvalues of a Hermitian matrix in non-increasing order, is always a convex polytope, known as a \emph{moment polytope}~\cite{ness-mumford,ASENS_1973_4_6_4_413_0,atiyah-convexity,guillemin-sternberg-convexity,kirwanthesis}!
Note that when restricted back to the commutative group setting, all the aforementioned concepts coincide with the ordinary ones (geodesically convex $\to$ convex, Riemannian gradient $\to$ gradient, moment polytope $\to$ Newton polytope).
Thus, efficient algorithms for solving unconstrained GP and the corresponding scaling problem are an important step towards a deeper understanding of the capacity problem of general (noncommutative) group actions.
This serves as an important motivation for our work.
In this spirit, we note that the box-constrained Newton methods of~\cite{allen2017much,cmtv-matrix-scaling,celis2019fair} have recently been generalized to general noncommutative capacity and scaling problems~\cite{azglow-operator-scaling,bfgoww-noncommutative-optimization}.


\subsection{Our results}
We provide two interior-point algorithms for unconstrained geometric programming and the corresponding scaling problem (\cref{prob:scaling zero,prob:scaling,prob:approximate GP,prob:approximate GP zero}).
The two algorithms differ in that the first algorithm assumes \emph{well-conditioned} instances, where the shift~$\theta$ is contained in the relative interior of the Newton polytope, while the second algorithm applies to any point in the Newton polytope.
For each algorithm, we give natural condition numbers associated with the input that allow us to tightly bound the number of iterations required to find an approximate solution.
For rational inputs, we accompany these results with a priori bounds for these condition numbers that are exponential in the encoding length of the input (resulting in polynomial-time algorithms), and we explain how to obtain polynomial bounds in special situations.
Our results improve over the optimization algorithms of~\cite{bfgoww-noncommutative-optimization}, which apply to general capacity and scaling problems, and our iteration complexities generalize recent results for matrix scaling and balancing~\cite{cmtv-matrix-scaling}.

\subsubsection{Well-conditioned instances}
We first discuss our results in the well-conditioned situation.
Here we use the following natural condition measures:

\begin{defn}[Geometric condition measures]
\label{def:condition numbers}
For an instance of the unconstrained GP or scaling problem with $\Omega = \{\omega_1, \dots, \omega_k\}\subseteq\R^n$ and shift $\theta \in \conv \Omega$, define
the \emph{distance from $\theta$ to the (relative) boundary of the Newton polytope} as
\begin{align}
\nonumber
  r_{\theta} &= \max \{ r \geq 0 : B(\theta, r) \cap \affspan \Omega \subseteq \conv\Omega \} = d(\theta, \partial\conv\Omega)
\intertext{where $B(\theta,r)$ denotes the closed ball centered at $\theta$ with radius $r$, and
$\affspan \Omega$ denotes the affine hull. 
Equivalently, $r_\theta$ is the radius of the largest ball about $\theta$ contained in the polytope.
Similarly, define the \emph{radius of the smallest enclosing ball about $\theta$}~as}
\label{eq:R_theta}
  R_{\theta}
&= \min \{ R \geq 0 : \conv\Omega \subseteq B(\theta, R) \}
= \max_{i\in[k]} \norm{\omega_i - \theta}_2.
\end{align}
We say that the instance is \emph{well-conditioned} if $r_\theta > 0$, otherwise it is called \textit{ill-conditioned}.
Thus, the instance is well-conditioned when $\theta$ is in the relative interior of the Newton polytope, and ill-conditioned if it is on the boundary.
\end{defn}

We remark that the ratio $R_\theta/r_\theta$ can be understood as a natural scale-invariant condition number of the instance.
See \cref{figure:geometric condition measures} for an illustration.

\begin{figure}
  \centering
  \includegraphics[height=3cm]{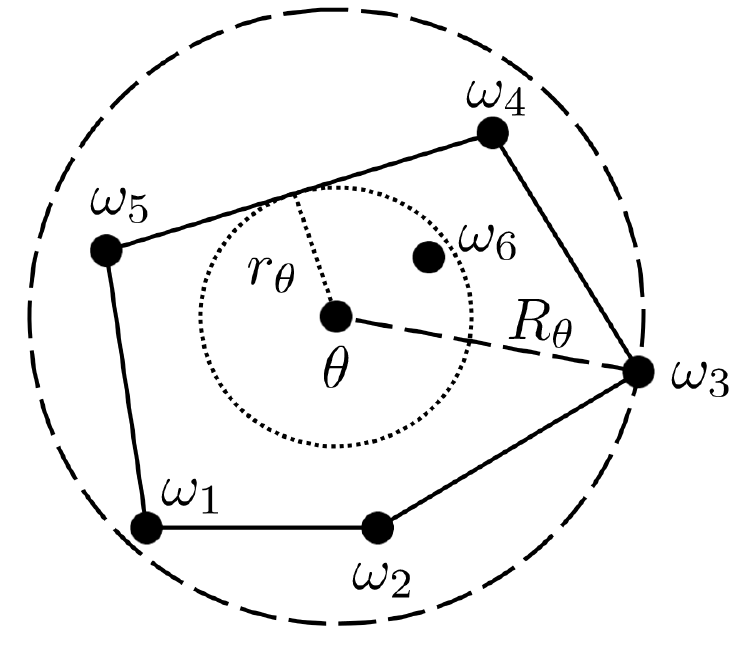}
  \caption{An illustration of $r_\theta$ and $R_\theta$, the distance from $\theta$ to the boundary and the radius of the smallest containing ball about $\theta$, respectively.}
  \label{figure:geometric condition measures}
\end{figure}

The geometric condition measure~$r_\theta$ is closely related to a condition measure due to Goffin~\cite{Goffin-measure}, which is widely used in the context of testing polyhedral cone feasibility (e.g.,~see~\cite{burgisser2013condition,doi:10.1287/moor.2019.1011}).
Let~$A\in \R^{n\times k}$ be a matrix with nonzero columns~$\vec{a}_1,\dots, \vec{a}_k$, and let $\hat a_i = a_i / \! \norm{a_i}_2$.
The \emph{Goffin measure} $\rho(A)$ of $A$ is defined as
\[
\rho(A)=\max_{\vec{x}\in\text{im}(A)\setminus\{\vec{0}\}}\min_{j\in[n]} \frac{\ip{\hat{\vec a}_j}{\vec{x}}}{\norm{\vec{x}}_2}.
\]
Its inverse has been called the GCC condition number~\cite[\S{}6.7]{burgisser2013condition}.
To see the connection, note that the distance to the boundary can be reformulated as
\[
  r_\theta = \min_{\vec{x} \in W \setminus \{0\}} \max_{i \in [k]} \, \frac{\ip{\omega_i - \theta}{\vec{x}}}{\norm{x}_2},
\]
where $W$ is the linear span of the $\omega_i - \theta$ in $\R^n$ (see \cref{lem:r_theta dual}).
Therefore, if every $\omega_i - \theta$ is a unit vector, then $r_\theta/R_\theta = r_\theta = - \rho(A)$,
where $A$ is the matrix with columns $a_i = \omega_i - \theta$. 
As such, $r_\theta$ can in general be viewed as an \emph{unnormalized} Goffin measure.
Indeed, while normalizing the vectors $\vec a_i$ does not affect the solvability of the corresponding conic feasibility problem, the relative sizes of the vectors $\omega_i - \theta$ directly influence the value of the corresponding geometric program, and so it is natural not to normalize in our setting.

Besides the geometric condition measures, we will also need the following quantity which captures the condition of the distribution~$q\in\R_{++}^k$:
\begin{align}\label{eq:beta}
  \beta = \frac{\norm{\vec q}_1}{\min_{i\in [k]} q_i},
\end{align}
where $\norm{\vec{q}}_1=\sum_{i=1}^k q_i$ denotes the $\ell_1$-norm.
In general, $k \leq \beta \leq \frac{\max_{i\in[k]} q_i}{\min_{i\in[k]} q_i} k < \infty$.
We note that since $\theta \in \conv \Omega$,
it holds that
\begin{align}\label{eq:lower bound optimal value}
  F_\theta^* \geq \log\min_{i\in[k]} q_i,
\end{align}
hence $F_\theta(0) - F_\theta^* \leq \log(\beta)$.

Our first result is a bound on the number of iteration steps of a natural interior-point method (IPM) which solves unconstrained GP for well-conditioned instances.

\begin{thm}\label{thm:algo wc}
\detailedboundsfalse\ugpthmwccontent
\end{thm}

We emphasize that it is \emph{not} necessary to provide a lower bound on $r_\theta$ as input.
The algorithm follows the interior-point method framework of~\cite{nesterov-nemirovskii,renegar-interior-point}, which consists of a \emph{preliminary stage} and a \emph{main stage}.
The preliminary stage uses a starting point that is easily computed in terms of the input data, and outputs a starting point for the main stage within $O( \sqrt k\log( k \, \frac{R_\theta}{r_\theta} \log(k\beta) ) )$ Newton iterations.
The main stage then produces a sequence of points $x_0, x_1, \dotsc$ such that $F_\theta(x_j) - F_\theta^* \leq C \log(k\beta) e^{-\frac{j c}{\sqrt{k}}}$ for some constants $c,C > 0$, implying the claimed iteration bound.

The same algorithm along with a known result relating the precision for geometric programming to the precision required to solve the scaling problem gives the following result (see \cref{subsec:scaling}).

\begin{cor}\label{cor:algorithm for scaling}
  There exists an algorithm that, given as input a well-conditioned instance of the scaling problem with shift (\cref{prob:scaling}), returns~$\vec{x}_{\epsilon} \in \R^n$ such that
  \begin{align*}
    \norm{\grad F_\theta(\vec{x}_\epsilon)}_2 = \norm{\grad F(\vec{x}_\epsilon) - \theta}_2 \leq \epsilon
  \end{align*}
  within
  \begin{align*}
    O\left(\sqrt k \log\left( k \frac{R_\theta}{r_\theta} \frac{R_\theta}{\epsilon} \log(k\beta) \right) \right)
  \end{align*}
  iterations.
\end{cor}

\subsubsection{General instances}
We now discuss our results for general instances (well-conditioned or not).
Here we provide an interior-point algorithm that approximates the unconstrained GP to arbitrary precision with an iteration complexity bound that is \emph{independent} of~$\theta$.
For this, we prove a $\theta$-independent diameter bound for approximate minimizers.
The quantity that controls our bound is the following.

\begin{defn}[Facet gap]\label{def:intro facet gap}
  Let $\Omega \subseteq \R^n$ be a finite set.
  The \emph{facet gap} $\varphi > 0$ of~$\Omega$ is the smallest distance from any $\omega \in \Omega$ to the affine span of any facet of $\conv \Omega$ not containing $\omega$.
  That is, it is the largest $\varphi > 0$ such that
  \begin{align*}
    d(\omega, \affspan F) \geq \varphi
  \end{align*}
   for any facet $F \subseteq \conv \Omega$ and $\omega \in \Omega \setminus F$.
\end{defn}

\begin{figure}
  \centering
  \includegraphics[height=3.5cm]{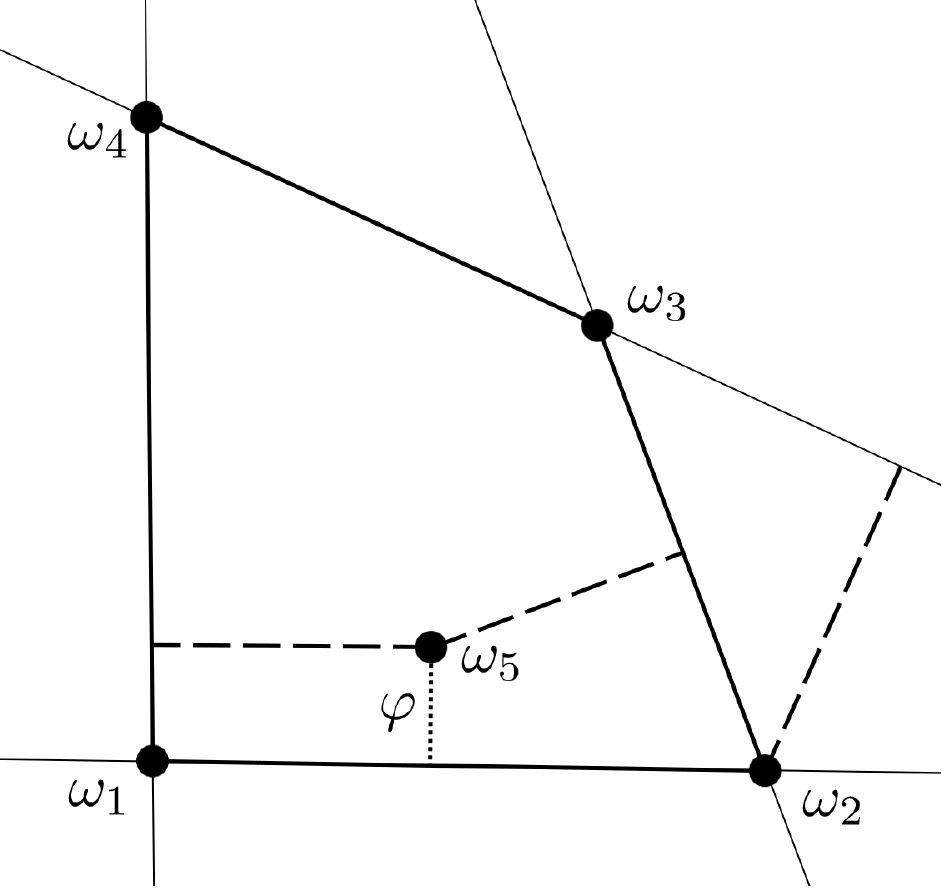}
  \caption{An illustration of the facet gap.
  For each facet of the Newton polytope, its affine hull is indicated by a solid thin line.
  The dashed line segments indicate the shortest distance between each affine hull and the closest $\omega_i$ not contained in the affine hull.
  The shortest line segment is dotted; its length is the facet gap~$\varphi$.}
  \label{figure:facet gap one}
\end{figure}

See \cref{figure:facet gap one} for an illustration.
The facet gap does not just depend on $\conv \Omega$, i.e., it is not a quantity that is purely controlled by the geometry of the Newton polytope.
We provide an example that shows that this is necessary for any diameter bound in the ill-conditioned setting (\cref{example:facet gap tight}).

Our diameter bound in terms of the facet gap (\cref{thm:facet gap radius bound}) generalizes a $\theta$-independent diameter bound obtained in~\cite{straszak-vishnoi-bitcomplexity} for \emph{integral} $\Omega \subseteq \Z^n$ to arbitrary $\Omega \subseteq \R^n$, with only small modifications to the proof.
The quantity that controls their diameter bound is called the \emph{unary facet complexity} of the Newton polytope, denoted by $\ufc$.
We recover their diameter bound by showing that, in the integral case, the facet gap and the unary facet complexity are related by~$\varphi^{-1} \leq \sqrt{n} \cdot \ufc$.
See \cref{section:condition and diameter} for details.
Finally, we denote the \emph{diameter} of the Newton polytope by
\begin{align}\label{eq:D}
  N = \max_{i\neq j} \norm{\omega_i - \omega_j}_2.
\end{align}
Our algorithmic result is then the following.

\begin{thm}\label{thm:algo general}
\detailedboundsfalse\ugpthmgeneralcontent
\end{thm}

\Cref{thm:algo general} applies to arbitrary points $\theta$ in the Newton polytope and achieves an iteration complexity that is fully independent of $\theta$.
In contrast, \cref{thm:algo wc} applies only to well-conditioned instances and its complexity is sensitive to the distance of~$\theta$ to the boundary of the Newton polytope.
However, the former algorithm relies crucially on an a priori lower bound on the facet gap of $\Omega$, while the latter has no such requirement.
As such, our two algorithmic results are incomparable.

As in the well-conditioned case, our algorithm also allows one to solve the scaling problem with a similar iteration complexity bound.

\begin{cor}\label{cor:algorithm for uniform scaling}
  There exists an algorithm that, given as input an instance of the scaling problem (\cref{prob:scaling}) with shift $\theta\in\conv\Omega$ as well as a lower bound $0 < \varphi_0 \leq \varphi$ on the facet gap, returns~$\vec{x}_{\epsilon} \in \R^n$ such that
  \begin{align*}
    \norm{\grad F_\theta(\vec{x}_\epsilon)}_2 = \norm{\grad F(\vec{x}_\epsilon) - \theta}_2 \leq \epsilon
  \end{align*}
  within
  \begin{align*}
    O\left( \sqrt{k} \log\left( k n \frac N {\varphi_0} \frac{R_\theta}{\epsilon} \log\left( \frac{k\beta R_\theta}{\epsilon} \right) \right) \right)
  \end{align*}
  iterations.
\end{cor}

The algorithm in \cref{cor:algorithm for uniform scaling} can also solve the weak membership problem for the Newton polytope or any other convex polytope given in V-representation~\cite{gls}.
Suppose one is given $\Omega \subseteq \Q^n$, $\theta \in \Q^n$, and $\epsilon > 0$, \emph{without} assuming that $\theta\in\conv\Omega$.
The weak membership problem asks to assert either that $d(\theta, \conv \Omega) \leq \epsilon$ or that the ball~$B(\theta, \epsilon)$ is not contained in $\conv \Omega$ (these conditions are not mutually exclusive).
One can run the algorithm from \cref{cor:algorithm for uniform scaling} with~$q = (1, \dotsc, 1) \in \R^k$, precision $\epsilon$, and the lower bound on the facet gap given below in~\eqref{eq:intro facet gap bit bound}.
If the algorithm does not terminate within the stated (polynomial) number of iterations, or if the returned point $x_\epsilon \in \R^n$ does not satisfy $\norm{\grad F(x_\epsilon) - \theta}_2 \leq \epsilon$, one may conclude that $\theta \not\in \conv \Omega$, hence $B(\theta, \epsilon)$ is not contained in $\conv \Omega$ either.
Otherwise, we obtain a point~$x_\epsilon \in \R^n$ such that $\norm{\grad F(x_\epsilon) - \theta}_2 \leq \epsilon$; since $\grad F(x_\epsilon) \in \conv \Omega$, one can therefore safely assert that $d(\theta, \conv \Omega) \leq \epsilon$.

\subsubsection{Bounds on condition measures}
Up to now, instances of the GP and scaling problems were allowed to be given by arbitrary real vectors.
We now discuss how our condition measures (and thereby the iteration complexity) can for \emph{rational} instances be effectively bounded in terms of the encoding length.
We will focus our attention on $r_\theta$ and $\varphi$ since the other condition measures~$\beta$, $R_\theta$, and $N$ can be straightforwardly bounded from their definition (see \cref{eq:beta,eq:R_theta,eq:D}).
Throughout, we follow the conventions of~\cite{gls} for the encoding length:
we encode rational numbers (and rational vectors) in binary, and write $\enclen{\cdot}$ for the encoding length.

We first consider the distance of $\theta$ to the boundary of the Newton polytope (\cref{def:condition numbers}).
In \cref{subsec:general bounds}, we prove the following upper bound on the distance to the boundary for well-conditioned rational instances:
\begin{align}\label{eq:upper bound gcn by bit complexity}
  \log_2 r_\theta^{-1} \leq 6n^2 \max_{i \in [k]} \, \enclen{\omega_i} + \enclen{\theta}. 
\end{align}
This bound implies that the iteration complexity of the interior-point method in \cref{thm:algo wc} is bounded by a polynomial in the encoding length of the instance.
In the (not necessarily commutative) setting of~\cite{bfgoww-noncommutative-optimization}, where $\Omega \subseteq \Z^n$ is integral, a first-order method was found which solves the scaling problem in $\poly(1/\epsilon, R_\theta)$ iterations.
They also developed a second-order method for the unconstrained GP problem based on the recently introduced notion of robustness whose iteration complexity is $\poly(\log(1/\delta), R_\theta/r_\theta, n)$ (cf.~\cite{allen2017much,cmtv-matrix-scaling,celis2019fair}).
Our results therefore improve upon both, as we have a logarithmic dependence on $1/\epsilon$ (for the scaling problem) and $1/\delta$ (for the geometric program), and logarithmic dependence on $R_\theta/r_\theta$.

We provide a bound similar to~\cref{eq:upper bound gcn by bit complexity} on the facet gap. For rational instances,
\begin{align}\label{eq:intro facet gap bit bound}
  \log_2 \varphi^{-1} \leq (6n^2 + 1) \max_{i \in [k]} \, \enclen{\omega_i}. 
\end{align}
This bound implies that the iteration complexity of the interior-point method in \cref{thm:algo general} is bounded by a polynomial in the encoding length of the instance (without any dependence on $\theta$).


The preceding bounds allow us to compare the performance of \cref{thm:algo wc} and \cref{thm:algo general} on well-conditioned instances.
We find that the second algorithm typically has lower iteration complexity, since it is independent of the encoding length of~$\theta$ while the dependence on the other parameters is comparable.
We note that the first algorithm has the advantage that by running the main stage for a sufficient number of iterations we can approximate the geometric program to arbitrary precision (see discussion below \cref{thm:algo wc}).
This is not the case for the second algorithm, which depends nontrivially on the desired accuracy.

\medskip

The general bounds on the condition measures in terms of the encoding length can be improved under a combinatorial hypothesis.
Let us say that an instance is \emph{totally unimodular} if the exponents $\omega_i$ are all integral and the matrix $A$ whose columns are given by the $\omega_i$ is totally unimodular, i.e., every subdeterminant is $\pm 1$ or $0$.
Then the facet gap can be \emph{polynomially} bounded as (see~\cref{thm:TUconditionbounds})
\begin{align}\label{eq:tu bound on facet gap}
  \varphi^{-1} \leq n^{3/2}.
\end{align}
We can similarly bound the distance to the boundary for well-conditioned instances as $r_\theta^{-1} \leq 2^{\enclen{\theta}} n^{3/2}$.
The latter bound appears in \cite{bfgoww-noncommutative-optimization} in the special case where $\theta = \vec{0}$; there, the quantity $r_{\theta}$ is known as the weight margin.
The total unimodularity also implies a polynomial bound on $R_\theta$ and $N$.
As such, for totally unimodular instances, the interior point algorithm in \cref{thm:algo general} can solve the unconstrained GP problem in $\tilde O(\sqrt k \log( \frac1\delta ) )$ iterations and the scaling problem in $\tilde O(\sqrt k \log( \frac1{\epsilon}))$ iterations.
The $\tilde{O}$ notation hides a poly(input length) term inside the logarithm.


As an important source of totally unimodular instances, suppose that $G$ is a directed graph with vertex set $V=[n]$, edge set~$E$ of size~$k$, and edge weights~$q_{ij}>0$ for $ij \in E$.
Since the incidence matrix of a directed graph is totally unimodular~\cite[\S{}19.3, Example 2]{schrijver1998theory}, the associated geometric program
\begin{align}\label{eq:quiver f}
  F_{G,\theta}(x) = \log \sum_{ij \in E} q_{ij} e^{x_i - x_j} - \ip \theta x
\end{align}
is totally unimodular.

Many widely studied applications fall into this setting.
For example, recall the matrix scaling problem from \cref{subsubsec:matrix scaling}.
If $G$ is a complete bipartite directed graph, then \cref{eq:quiver f} recovers the matrix scaling objective function in \cref{eq:matrix scaling posy} up to a trivial change of variables $y\to -y$.
Thus our interior point algorithm requires $\tilde O(\sqrt k \log( \frac1{\epsilon}))$ iterations for finding an $\epsilon$-approximate $(\vec r, \vec c)$-scaling of a non-negative matrix with $k$ non-zero entries (if such a scaling exists).
The matrix balancing problem can similarly be modeled by taking~$G$ to be a complete directed graph.
Unconstrained GPs of the form \cref{eq:quiver f} can in general be related to nonlinear flow problems on directed graphs~\cite{cmtv-matrix-scaling}.

The iteration complexity that we obtain for matrix scaling and balancing slightly improves over (but is essentially the same as) the one given in~\cite{cmtv-matrix-scaling} for an interior-point method designed specifically for these problems.
It is natural to ask whether we can also meet the time complexity of the latter, which relied on a slightly different objective function and a clever implementation of approximate Newton iterations by using Laplacian solvers.
We leave this question for future investigation.

\subsection{Organization of the paper}
The rest of this paper is organized as follows.
In \cref{section:condition and diameter} we discuss the condition numbers defined in~\cref{def:condition numbers,def:intro facet gap} in detail and show how they imply diameter bounds on (approximate) minimizers of the GP.
In \cref{section:ipm and complexity} we explain how to use these diameter bounds together with the general framework of interior-point methods to prove \cref{thm:algo wc,thm:algo general} and their corollaries.
In \cref{section:total unimodularity}, we give a priori bounds on the condition numbers in terms of the encoding length of the input and we also provide better condition number bounds when the geometric program is totally unimodular.

\subsection{Acknowledgments}
We would like to thank Daniel Dadush for inspiring discussions, and Sander Gribling for helpful feedback.
MW acknowledges support through an NWO Veni grant no.~680-47-459.
PB was partially funded by the European Research Council (ERC) under the European's Horizon 2020 research and innovation programme (grant agreement No 787840).

\section{Condition measures and diameter bounds}\label{section:condition and diameter}
In this section, we discuss the condition measures defined in \cref{def:condition numbers,def:intro facet gap} in more detail and show how they imply diameter bounds on the geometric program.
Throughout, we fix an instance of the unconstrained GP or scaling problem with $\Omega = \{\omega_1, \dotsc, \omega_k\} \subseteq \R^n$, $\vec{q} \in \R_{++}^k$, and shift~$\theta \in \conv \Omega$.
Recall from \cref{eq:log gp shifted} that the objective function $F_\theta\colon \R^n \to \R$ is given by
\begin{align*}
  F_\theta(\vec{x})
= \log \sum_{i=1}^k q_ie^{\ip{\omega_i - \theta}{\vec{x}}}.
\end{align*}
We denote its infimum by $F_\theta^* = \inf_{\vec{x} \in \R^n} F_\theta(\vec{x})$.
The promise that $\theta \in \conv \Omega$ guarantees that $F_\theta^*$ is finite.
We further note that
\begin{align}\label{eq:f stab}
  F_\theta(x) = F_\theta(x+h) \qquad \forall h \in W^\perp,
\end{align}
where $W$ denotes the linear span of the $\omega_i - \theta$ in $\R^n$.
Since $\theta\in\conv\Omega$, this is the same as the direction vector space of $\affspan\Omega$ and hence independent of $\theta$.
\Cref{eq:f stab} implies that we can restrict the optimization problem to $W$.

\subsection{Well-conditioned instances}
We first consider the case that the instance is well-conditioned.
Recall that this means that $\theta$ is in the relative interior of the Newton polytope, so $r_{\theta} > 0$, where
\begin{align*}
  r_\theta = d(\theta, \partial\conv\Omega) = \max \{ r \geq 0 : B(\theta, r) \cap \affspan\Omega \subseteq \conv\Omega \}
\end{align*}
denotes the distance from $\theta$ to the (relative) boundary of the Newton polytope, as in \cref{def:condition numbers}.
We have the following useful dual expression.

\begin{lem}\label{lem:r_theta dual}
Let $\theta \in \relint\conv\Omega$.
Then,
\[
  r_\theta
= \min_{\substack{\vec{x} \in W \\ \norm x_2=1}} \max_{u \in \conv\Omega} \, \ip{u - \theta}{\vec{x}}
= \min_{\vec{x} \in W \setminus \{0\}} \max_{i \in [k]} \frac {\ip{\omega_i - \theta}{\vec{x}}} {\norm x_2}.
\]
\end{lem}
\begin{proof}
The second equality is clear.
To see the first equality, note that for any vector $x\in W$ with $\norm{x}_2=1$, $\max_{u \in \conv\Omega} \, \ip{u - \theta}{\vec{x}}$ is the distance from $\theta$ to the face of $\conv \Omega$ determined by the vector $x$.
Now minimizing over all such~$x$ results in the shortest distance from $\theta$ to any face of the polytope $\conv \Omega$, that is, in~$r_\theta$.
\end{proof}

Alternatively, \cref{lem:r_theta dual} can be proved by using polar duality.
We now state our diameter bound for well-conditioned instances.
In addition to the distance~$r_\theta$, we also need the quantity
\begin{align*}
  \beta = \frac{\norm{\vec q}_1}{\min_{i\in [k]} q_i},
\end{align*}
from \cref{eq:beta} which captures the condition of the coefficients~$q$ in the GP.

\begin{lem}\label{lem:wc large x}
  Let $\theta\in\relint\conv\Omega$.
  For all $x\in W$ with norm $\norm x_2 > \frac{\log(c\beta)}{r_\theta}$, where $c>0$, there exists $i_0 \in [k]$ such that
  \begin{align*}
    q_{i_0} e^{\ip{\omega_{i_0} - \theta} x} > c \norm q_1.
  \end{align*}
  In particular, $F_\theta(x) > F_\theta(0) + \log c$.
\end{lem}
\begin{proof}
  Let $\vec{x} \in W$ be such that $\norm x_2 > \frac{\log(c\beta)}{r_\theta}$.
  By \cref{lem:r_theta dual},
  \begin{align*}
    r_\theta \leq \max_{i\in[k]} \frac{\ip{\omega_i - \theta} x}{\norm x_2},
  \end{align*}
  and hence there exists $i_0 \in [k]$ such that
  \begin{align*}
    \ip{\omega_{i_0} - \theta} x \geq r_\theta \norm x_2 > \log(c\beta).
  \end{align*}
  This establishes the first claim, since now
  \begin{align*}
    q_{i_0} e^{\ip{\omega_{i_0} - \theta} x} > q_{i_0} e^{\log(c\beta)} = c q_{i_0} \beta \geq c \norm q_1.
  \end{align*}
  The second claim follows from this, since $e^{F_\theta(x)} \geq q_{i_0} e^{\ip{\omega_{i_0} - \theta} x}$ and $e^{F_\theta(0)} = \norm q_1$.
\end{proof}

\begin{cor}[Well-conditioned diameter bound]\label{cor:WCdiameterbound}
  Let $\theta\in\relint\conv\Omega$.
  Then there exists $x\in W$ with $\norm x_2 \leq \frac{\log\beta}{r_\theta}$ such that
  \begin{align*}
    F_\theta(\vec{x}) = F_\theta^*.
  \end{align*}
\end{cor}

In the case where $\Omega \subseteq \{0,1\}^n$ and $q=(1,\dots,1)$ is the all-ones vector, the diameter bound $\norm{\vec x}_2 \leq \frac n {r_\theta}$ was obtained in~\cite{Singh-Vishnoi-14}.
Ref.~\cite{straszak-vishnoi-bitcomplexity} improves this bound to $\norm{\vec x}_2 \leq \frac{\log k} {r_\theta}$ for general~$\Omega$ but the same~$q$ (a special case of \cref{cor:WCdiameterbound}).

\subsection{General instances}
In this subsection we allow $\theta$ to be an arbitrary point in the Newton polytope.
Here, the central quantity is the \emph{facet gap} of $\Omega$, which we recall was defined as the largest constant $\varphi>0$ such that
\begin{align*}
    d(\omega, \affspan F) \geq \varphi
\end{align*}
for any facet $F \subseteq \conv \Omega$ and $\omega \in \Omega \setminus F$ (see \cref{def:intro facet gap}).

The following theorem improves upon \cite[Thm.~4.1]{straszak-vishnoi-bitcomplexity}, as we will discuss below \cref{prop:facet gap ufc}.
Its proof follows essentially the same argument, with a slight modification that also avoids the recursion and leads to a slightly better bound.

\begin{thm}[Diameter bound via facet gap]\label{thm:facet gap radius bound}
  For any $0<\delta<2\beta$ and $\theta\in\conv\Omega$, there exists $x\in W$ such that
  \begin{align*}
    \norm{x}_2 \leq \frac{m}\varphi \log\left( \frac{2 \beta}\delta \right)
  \end{align*}
  and
  \begin{align*}
    F_\theta(\vec{x}) \leq F_\theta^* + \delta,
  \end{align*}
  where $m = \dim\affspan\Omega \leq n$.
\end{thm}
\begin{proof}
  To start, choose vectors $\vec{a}_j \in W$ with $\norm{\vec{a}_j}_2 = 1$ and scalars $b_j \in \R$ for $j \in J$ some finite index set, such that the Newton polytope is defined by
  \begin{align*}
    \conv \Omega = \left\{ \vec{p} \in \affspan \Omega : \ip p {a_j} \leq b_j \quad \forall j \in J \right\}.
  \end{align*}
  We assume each inequality defines a facet of the polytope.
  Define the \emph{normal cone}~$N_\omega$ at a vertex $\omega$ to be $N_\omega = \{ \sum_{j \in J_\omega} c_j a_j : c_j \geq 0 \}$ where $J_\omega = \{ j \in J : \ip{a_j}{\omega} = b_j \}$ is the set of tight constraints at $\omega$.
  It is well-known that $W = \bigcup_\omega N_\omega$, where $\omega$ ranges over the vertices of $\conv \Omega$ (the normal fan is complete).

  Now fix $\theta \in \conv \Omega$ and let $\vec{x}^* \in W$ be such that
  \begin{align*}
    F_\theta(\vec{x}^*) \leq F_\theta^* + \frac{\delta}{2}.
  \end{align*}
  Then~$x^* \in N_{\omega'}$ for some vertex $\omega' \in \Omega$ of $\conv\Omega$, hence there exists a subset~$J' \subseteq J_{\omega'} \subseteq J$ and non-negative numbers~$\{c_j\}_{j \in J'}$ such that $x^* = \sum_{j \in J'} c_j a_j$. By Carath\'eodory's theorem, we may assume~$|J'| \leq m = \dim W$.
  Now define
  \begin{align*}
      \Delta := \frac1\varphi \log\left( \frac{2 \beta} \delta \right),
  \end{align*}
  which is positive by the assumption that $\delta<2\beta$, and set
  \begin{align*}
    \vec{x} := \sum_{j \in J'} \min(c_j, \Delta) \vec{a}_j.
  \end{align*}
  Since $\norm{\vec{a}_j}_2 = 1$, we have $\norm{\vec{x}}_2 \leq \abs{J'} \cdot \Delta \leq m \Delta$, so $x$ satisfies the desired norm bound.

  To complete the proof, it therefore suffices to show that
  \begin{align}\label{eq:goal}
    F_\theta(\vec{x}) \leq F_\theta(\vec{x}^*) + \frac{\delta}{2}.
  \end{align}
  We start by setting $c_j' = \min(c_j, \Delta)$ for convenience, so that $x = \sum_{j\in J'} c'_j a_j$.
  Let~$J_0$ consist of those $j \in J'$ such that $c_j' \neq c_j$, i.e., $c_j > \Delta$ and $c'_j = \Delta$.
  We may assume there exists at least one $j_0 \in J_0$; otherwise, $c_j \leq \Delta$ for every $j \in J'$, so we have $x = x^*$ and \cref{eq:goal} holds trivially.
  Now consider the intersection of $\Omega$ with the face defined by the constraints~$J'$,
  \begin{align*}
    \Omega' = \{ \omega \in \Omega : \ip{\vec{a}_j}{\omega} = b_j \quad \forall j \in J' \},
  \end{align*}
  If $\omega\in\Omega\setminus\Omega'$, then $\omega\not\in\affspan\Omega'$, and
  \begin{align*}
    \ip \omega x - \ip {\omega'} x
  &= \sum_{j \in J'} c'_j \ip {\omega-\omega'} {a_j}
  = \sum_{j \in J'} c'_j \left( \ip \omega {a_j} - b_j \right) \\
  &\leq c'_{j_0} \left( \ip \omega {a_{j_0}} - b_{j_0} \right)
  = \Delta \left( \ip \omega {a_{j_0}} - b_{j_0} \right) \\
  &\leq -\varphi\Delta
  = \log \left(\frac\delta{2\beta}\right).
  \end{align*}
  The first inequality holds since each term in the sum is non-positive.
  The second inequality follows from the observation that the distance from $\omega$ to the affine span of the facet defined by $a_{j_0}$ and $b_{j_0}$ is
  \begin{align*}
    \frac{b_{j_0} - \ip \omega {a_{j_0}}}{\norm{a_{j_0}}_2}
  = b_{j_0} - \ip \omega {a_{j_0}}
  \geq
  \varphi
  \end{align*}
  by definition of the facet gap.
  So we obtain for $\omega\in\Omega\setminus\Omega'$ that
  \begin{align}\label{eq:unprime bound}
    \beta e^{\ip {\omega-\omega'} x} \leq \frac\delta2.
  \end{align}
  On the other hand, if $\omega \in \Omega'$, then
  \begin{align}\label{eq:prime bound}
    \ip{\omega - \theta}{\vec{x}} & = \ip{\omega - \theta}{\vec{x}^*} - \sum_{j \in J_0} (c_j - c'_j) \ip{\omega - \theta}{\vec{a}_j} \leq \ip{\omega - \theta}{\vec{x}^*}
  \end{align}
  since $c_j \geq c_j'$ and $\ip{\theta}{\vec{a}_j} \leq b_j = \ip{\omega}{\vec{a}_j}$ for all $j\in J'$.
  Therefore, we now obtain
  \begin{align*}
    F_\theta(x) 
  &= \log \left( \sum_{i : \omega_i \in \Omega'} q_i e^{\ip{\omega_i - \theta} x} \right)
   + \log \left( 1 + \frac{\sum_{i : \omega_i \not\in \Omega'} q_i e^{\ip{\omega_i - \theta} x}}{\sum_{i : \omega_i \in \Omega'} q_i e^{\ip{\omega_i - \theta} x}} \right) \\
  &\leq \log \left( \sum_{i : \omega_i \in \Omega'} q_i e^{\ip{\omega_i - \theta} x} \right)
   + \frac{\sum_{i : \omega_i \not\in \Omega'} q_i e^{\ip{\omega_i - \theta} x}}{\sum_{i : \omega_i \in \Omega'} q_i e^{\ip{\omega_i - \theta} x}} \\
  &\leq \log \left( \sum_{i : \omega_i \in \Omega'} q_i e^{\ip{\omega_i - \theta} x} \right)
   + \frac{\sum_{i : \omega_i \not\in \Omega'} q_i e^{\ip{\omega_i - \theta} x}}{q_{i'} e^{\ip{\omega' - \theta} x}} \\
  &\leq \log \left( \sum_{i : \omega_i \in \Omega'} q_i e^{\ip{\omega_i - \theta} {x^*}} \right)
   + \frac\delta 2 \\
  &\leq F_\theta(x^*) + \frac\delta2.
  \end{align*}
  In the second inequality we denote by $i'\in[k]$ an index such that $\omega_{i'} = \omega'$, observing that $\omega' \in \Omega'$.
  The third inequality follows from \cref{eq:unprime bound,eq:prime bound}.
\end{proof}

In contrast with the diameter bound for well-conditioned instances, which is in terms of the distance of~$\theta$ to the boundary of the Newton polytope, the diameter bound in \cref{thm:facet gap radius bound} is \emph{independent} of the shift~$\theta$.
However, the facet gap is not an intrinsic property of the Newton polytope~$\conv\Omega$ but rather depends on the entire set of exponents~$\Omega$, so the same is true for the diameter bounds in terms of the facet gap.
The following example shows that this is necessary.

\begin{example}\label{example:facet gap tight}
  Consider for $\varphi\in(0,1/2]$ the instance with $\Omega = \{0,\varphi,1\} \subseteq \R$, $q=(1,1,1) \in \R^3$, and $\theta=0$.
  It is clear that $\conv\Omega = [0,1]$ and that $\Omega$ has facet gap equal to~$\varphi$.
  Furthermore, we have
  \begin{align*}
    F_\theta(x) = \log\left( 1 + e^{\varphi x} + e^x \right) \geq 0
  \end{align*}
  and $\lim_{x\to-\infty} F_\theta(x) = 0$, so $F^*_\theta = 0$.
  On the other hand,
  \begin{align*}
    F_\theta(x) \geq \log\left( 1 + e^{\varphi x} \right),
  \end{align*}
  so any $\delta$-approximate minimizer for $\delta\in(0,1)$ must satisfy $\abs x \geq -x \geq \frac1\varphi\log \frac1{2\delta}$.
\end{example}

We now explain how \cref{thm:facet gap radius bound} implies the diameter bound of \cite{straszak-vishnoi-bitcomplexity,celis2019fair}.
We first recall a definition from \cite{straszak-vishnoi-bitcomplexity}.

\begin{defn}[Unary facet complexity]\label{def:ufc}
  Let $P \subseteq \R^n$ be an integral polytope.
  The \textit{unary facet complexity} $\ufc(P)$ is the smallest integer $M \geq 0$ such that $P$ can be described as the intersection of the affine span of $P$ with half-spaces~$\ip p a \leq b$, where $a \in \Z^n$, $b \in \R$, and $\norm{a}_{\infty} \leq M$.
\end{defn}

The following proposition shows that the facet gap can be bounded in terms of the unary facet complexity.

\begin{prop}\label{prop:facet gap ufc}
  Let $\Omega \subseteq \Z^n$.
  Then the facet gap~$\varphi$ of $\Omega$ satisfies
  \begin{align*}
    \frac1\varphi \leq \sqrt{n} \cdot \ufc(\conv \Omega).
  \end{align*}
\end{prop}
\begin{proof}
  For any facet $F \subset \conv\Omega$ there exists a corresponding half-space $\ip \cdot a \leq b$ defined by $a\in\Z^n$, $b\in\R$, and $\norm a_\infty \leq \ufc(\conv\Omega)$.
  Then the affine span of the facet is given by $\affspan F = \affspan\Omega \cap H$, where $H$ is the affine hyperplane
  \begin{align*}
    H = \left\{ p \in \R^n : \ip a p = b \right\}.
  \end{align*}
  As a consequence, the distance from any $\omega \in \Omega \setminus F$ to $\affspan F$ can be lower bounded by the distance of $\omega$ to the affine hyperplane~$H$, that is,
  \begin{align*}
    d(\omega, \affspan F)
  \geq \frac{b - \ip a \omega}{\norm a_2}
  = \frac{\ip a {\omega'} - \ip a \omega}{\norm a_2}
  \geq \frac1{\sqrt n \cdot \ufc(\conv\Omega)},
  \end{align*}
  where $\omega'$ is an arbitrary point in~$\Omega \cap F$.
  To see the inequality, note that the numerator is positive and an integer since $a, \omega, \omega' \in \Z^n$, so at least $1$, whereas the denominator is at most $\sqrt n \cdot \ufc(\conv\Omega)$.
\end{proof}

Thus, \cref{thm:facet gap radius bound,prop:facet gap ufc} imply the following diameter bound:
For integral~$\Omega\subseteq\Z^n$, there exists a $\delta$-approximate minimizer of $F_\theta$ of norm
\begin{align*}
  \norm{x}_2
\leq n^{3/2} \ufc(\conv\Omega) \log\left( \frac{2 k \frac{\max q_i}{\min q_i}}\delta \right)
\leq n^{3/2} \ufc(\conv\Omega) \left( 2 L_p + \log\left( \frac{2 k}\delta \right) \right),
\end{align*}
where $L_p = \max_i \abs{\log q_i}$ and we used that $\beta \leq k \frac{\max q_i}{\min q_i} \leq k e^{2 L_p}$.
The right-hand side bound is essentially the original diameter bound from~\cite{straszak-vishnoi-bitcomplexity} with a logarithmically improved dependence on~$n$.
The middle bound is very similar to a bound stated in an older version of~\cite{celis2019fair}.

\section{Interior-point methods and complexity analysis}\label{section:ipm and complexity}
In this section, we show that approximate minimizers of $F_\theta$ may be found efficiently using interior-point methods (IPM).
The idea is to rewrite the geometric program as a \textit{linear} optimization over a (more complicated) convex domain, for which we know an explicit \textit{self-concordant barrier functional}.
The domain and the corresponding barrier will be slightly different in the well-conditioned and the general case.

\subsection{Background: the barrier method}
We recall the standard interior-point formalism for solving a convex program
\begin{equation}\label{eq:convex prog}
\begin{split}
\text{minimize} \quad & \ip{\vec c}{\vec p} \\
  \text{subject to} \quad & \vec p \in D,
\end{split}
\end{equation}
where $\vec{c} \in E$ and $D\subseteq E$ is a closed convex set of some Euclidean space $E$, and $\ip{\cdot}{\cdot}$ is the inner product on $E$.
We follow the exposition in~\cite{renegar-interior-point} and omit most proofs;
another source on interior-point methods is~\cite{nesterov-nemirovskii}.

Let us first fix some notation and language.
We denote by $\val$ the optimal value of~\cref{eq:convex prog}, and we say $\vec{p}_\delta \in D$ is a $\delta$-minimizer of~\cref{eq:convex prog} if
\begin{align*}
  \ip{\vec c}{\vec{p}_\delta}\leq\val+\delta.
\end{align*}
We say a functional $\Psi\colon\interior(D)\to \R$ is twice continuously differentiable if its gradient and Hessian, which we shall always denote by
\begin{align*}
  g(\vec{p})=\grad\Psi(\vec{p}), \qquad H(\vec{p})=\grad^2\Psi(\vec{p}),
\end{align*}
are well-defined at any point $\vec{p}\in \interior(D)$, and $H(\vec{p})$ depends continuously on $\vec p$.
Recall that such a function is strictly convex if for any $\vec{p}\in D$, the Hessian $H(\vec{p})$ is positive definite.
In this case, the Hessian $H(\vec p)$ defines a \emph{local norm} on~$E$ for any~$\vec p \in D$:
for $v \in E$, we write
\begin{align*}
  \norm{v}_{(p)} = \sqrt{\ip{v}{H(\vec{p}) v}} \, .
\end{align*}
We also write $B_p^\circ(\vec{p},1) = \{ p' \in E : \norm{p'-p}_{(p)} < 1 \}$ for the \emph{open} ball with radius~$1$ centered at some point~$p\in D$, measured in the local norm $\norm{\cdot}_{(p)}$ at the same point.
The following notion is central in interior-point methods.

\begin{defn}
  Let $D \subseteq E$ be closed convex with non-empty interior.
  A \emph{(strongly non-degenerate) self-concordant barrier functional} for $D$ is a strictly convex and twice continuously differentiable function $\Psi\colon \interior(D) \to \R$, satisfying the following additional properties:
  \begin{enumerate}
  \item For any $\vec{p}\in\interior(D)$, the open ball $B_\vec{p}^\circ(\vec{p},1)$ is contained in $\interior(D)$.
  Moreover, for any $\vec{p}'\in B_\vec{p}^\circ(\vec{p},1)$, we have
  \begin{align*}
    1 - \norm{p' - p}_{(p)} \leq \frac {\norm v_{(p')}} {\norm v_{(p)}} \leq \frac 1 {1 - \norm{p' - p}_{(p)}}
    \quad \text{for all } \vec{v} \in E \setminus \{\vec{0}\}.
  \end{align*}
  \item The \emph{complexity parameter}~$\nu$ of the barrier, defined by
    \begin{equation*}
      \nu:=\sup \, \{ \norm{H(p)^{-1} g(p)}_{(p)}^2 \;:\; p \in \interior(D) \},
    \end{equation*}
    is finite.
  \end{enumerate}
\end{defn}

One can show that a self-concordant barrier $\Psi$ blows up at the boundary of its domain:
if $p_i\in \interior(D)$ converges to $\overline{p}\in\partial D$, then $\Psi(p_i)\to\infty$ and $\norm{g(p_i)}_2\to\infty$ for~$i\to\infty$; see~\cite[Thm~2.2.9]{renegar-interior-point}.
An instructive example is the barrier $\Psi(p) = -\log p$ for the half-line $\R_+ \subseteq \R$, which has complexity parameter $\nu=1$.
Self-concordance can also be defined in terms of an estimate of the third derivatives by the second derivatives.

We now return to the convex program in \cref{eq:convex prog} and suppose that the domain~$D$ is bounded and admits a self-concordant barrier functional~$\Psi$ as above.
Standard barrier methods consist of a \emph{main stage} and a \emph{preliminary stage}.
In the main stage, one follows the \emph{central path}, which consists of the minimizers $\vec{z}(\eta) \in \interior(D)$ of the self-concordant functionals
\begin{align*}
  \Psi_\eta(\vec{p}):=\eta\ip{\vec{c}}{\vec{p}}+\Psi(\vec{p})
\end{align*}
for every $\eta\in\R_{++}$.
These minimizers exist, since the domain is bounded and $\Psi$~blows up at its boundary, and they are unique, since $\Psi$ is strictly convex.
It is well-known (cf.~\cite[(2.12)]{renegar-interior-point}) that
\begin{align}
  \label{eq:central path estimate}
  \ip{\vec{c}}{\vec{z}(\eta)}\leq\val+\frac{\nu}{\eta},
\end{align}
so following the central path as $\eta \to \infty$ guarantees convergence of $\vec{z}(\eta)$ to a minimizer $\vec{z}$ of the objective.
The preliminary stage is used beforehand to find a point sufficiently close to the central path to start the main stage.

The main stage is described in~\cref{algo:IPM main stage}.
One assumes to have a starting parameter $\eta_0$, and a starting point $\vec{p}_0 \in \interior(D)$, which is an approximate minimizer of $\Psi_{\eta_0}$, so in particular close to $\vec{z}(\eta_0)$.
Then, for $i \geq 1$, one chooses an appropriate~$\eta_i > \eta_{i-1}$ such that a single Newton step for the function $\Psi_{\eta_i}$ at point~$p_{i-1}$ produces a point $p_i$ that is guaranteed to remain close to the central path.
Since $\grad \Psi_{\eta_i}(p) = \eta_i c + \grad \Psi(p) = \eta_i c + g(p)$ and $\grad^2 \Psi_{\eta_i}(p) = \grad^2 \Psi(p) = H(p)$,
the point~$p_i$ obtained by taking a single Newton step is given by
\begin{align*}
  p_i = p_{i-1} - (\grad^2 \Psi_{\eta_i}(p_{i-1}))^{-1} \grad \Psi_{\eta_i}(p_{i-1}) = p_{i-1} - H(p_{i-1})^{-1} (\eta_i c + g(p_{i-1})).
\end{align*}
If we write
\begin{align*}
  \alpha_i(p) := \norm{H(p)^{-1}(\eta_i c + g(p))}_{(p)},
\end{align*}
then the length of the Newton step, measured in the local norm at $p_{i-1}$, is $\alpha_i(p_{i-1})$.
Furthermore, one can show that~$\alpha_i(p)$ is directly related to the distance of~$p$ to the minimizer~$z(\eta_i)$ of~$\Psi_{\eta_i}$ (cf.~\cite[Thm.~2.2.5]{renegar-interior-point}).
Therefore, by choosing the~$\eta_i$ such that~$\alpha_i(p_i)$ stays small, we guarantee that the iterates~$p_i$ remain close to the central path.
This is achieved by first estimating~$\alpha_i(p_{i-1})$ in terms of~$\alpha_{i-1}(p_{i-1})$, the ratio~$\eta_i / \eta_{i-1}$, and the complexity parameter~$\nu$ of the barrier, and then bounding~$\alpha_i(p_i)$ in terms of~$\alpha_i(p_{i-1})$ using self-concordance \cite[(2.15)--(2.16)]{renegar-interior-point}.
Provided~$\eta_i \to \infty$ as~$i \to \infty$, \cref{eq:central path estimate} suggests that the~$p_i$ converge to a minimizer of the objective.
A suitable choice of the~$\eta_i$, along with a quantitative guarantee on the precision achieved by any particular~$\vec{p}_i$ is given by the following theorem.

\begin{algorithm}[t]
  \Input{starting point $\vec p_0 \in D$, starting parameter $\eta_0 > 0$, objective $\vec c \in E$, iteration count $T \geq 0$, complexity parameter $\nu \geq 1$ and oracle access to gradient $g(\vec p)$ and Hessian $H(\vec p)$ of barrier $\Psi\colon\interior(D)\to\R$}
  \BlankLine
  \For{$i = 1, \dotsc, T$}{
    $\eta_i \leftarrow \left( 1 + \frac{1}{8 \sqrt{\nu}} \right) \eta_{i-1}$\;
    $\vec{p}_{i} \leftarrow \vec{p}_{i-1} - H(\vec{p}_{i-1})^{-1} \left( \eta_{i} \vec{c} + g(\vec{p}_{i-1}) \right)$ \Comment*{Newton step for $\Psi_{\eta_i}$ at $p_{i-1}$}
  }
  \Return{$\vec{p}_T$}\;
  \caption{MainStage}
  \label{algo:IPM main stage}
\end{algorithm}

\begin{thm}[Main stage, {\cite[p.~46 and~(2.14)]{renegar-interior-point}}]\label{thm:iteration complexity of IPM main stage}
Let $\Psi\colon \interior(D) \to \R$ be a strongly non-degenerate self-concordant barrier functional for $D$ with complexity parameter $\nu \geq 1$.
Let $\eta_0 > 0$ be given, and suppose $\vec{p}_0 \in \interior(D)$ satisfies
\begin{equation}\label{eq:starting condition}
  \alpha_0(p_0) = \norm{H(\vec{p}_0)^{-1}(\eta_0\vec{c}+g(\vec{p}_0))}_{(p_0)} \leq \frac{1}{9}.
\end{equation}
Then the iterations of \cref{algo:IPM main stage} are well-defined and we have, for all $i\in[T]$, that $\alpha_i(p_i) \leq \frac{1}{9}$, $\norm{p_i - z(\eta_i)}_{(z(\eta_i))} \leq \frac 1 5$, and $\ip c {p_i} \leq \val + \frac6{5\eta_i} \nu$.

In particular, for $T \geq 10 \sqrt{\nu} \log ( \frac{6}{5} \frac{\nu}{\eta_0 \delta} )$, \cref{algo:IPM main stage} returns a point $p_T \in \interior(D)$ satisfying
\begin{align*}
  \ip{\vec{c}}{\vec{p}_T} \leq \val + \delta.
\end{align*}
\end{thm}

The goal of the preliminary stage is to find a starting point and a starting parameter satisfying the hypotheses of \cref{thm:iteration complexity of IPM main stage}.
The algorithm is presented in \cref{algo:IPM pre stage}.
One starts from an arbitrary point $p_0'\in \interior(D)$ and follows the central path associated with the objective $-\vec g(\vec{p}'_0)$ and the same self-concordant barrier.
This objective is chosen because $\vec{p}_0'$ is the minimizer of $-\mu\ip{\vec g(\vec{p}_0')}{\vec{p}} + \Psi(\vec{p})$ when $\mu=1$, i.e., $\vec{p}_0'$ is exactly on the central path at time $1$.
Now one \emph{decreases} the parameter $\mu$, rather than increasing it, until one obtains an approximate minimizer of~$\Psi = \Psi_0$.
Finally, one chooses an appropriate $\eta_0>0$ and performs a single Newton step for~$\Psi_{\eta_0}$ that is guaranteed to yield a point $p_0$ satisfying \cref{eq:starting condition}.
Only this last step depends on the objective~$c$ of the convex program~\eqref{eq:convex prog}.
The following definition and theorem bound the number of iterations of \cref{algo:IPM pre stage} and give a lower bound on~$\eta_0$.

\begin{algorithm}[t]
  \Input{starting point $\vec{p}_0' \in D$,
    objective $\vec c \in E$,
    complexity parameter $\nu \geq 1$ and oracle access to gradient $g(\vec p)$ and Hessian $H(\vec p)$ of barrier $\Psi\colon\interior(D)\to\R$
  }
  \BlankLine
  $\mu_0 \leftarrow 1$\;
  $\vec g_0 \leftarrow g(\vec{p}_0')$\;
  $i \leftarrow 0$\;
  \While{$\norm{H(\vec{p}_i')^{-1}g(\vec{p}_i')}_{(p_i')} > \frac{1}{6}$}{
    $i \leftarrow i+1$\;
    $\mu_i \leftarrow \left( 1 - \frac{1}{8\sqrt{\nu}} \right) \mu_{i-1}$\;
    $\vec{p}_{i}' \leftarrow \vec{p}_{i-1}' - H(\vec{p}_{i-1}')^{-1} (-\mu_{i} \vec{g}_0 + g(\vec{p}_{i-1}'))$\Comment*{Newton step for $-\mu_i g_0 \!+\! \Psi$ at $p_{i-1}'$}
  }
  $\eta_0 \leftarrow (12 \norm{H(p_i')^{-1} c}_{(p_i')})^{-1}$ \;
  $p_0 \leftarrow p_i' - H(p_i')^{-1} (\eta_0 c + g(p_i'))$ \Comment*{Newton step for $\Psi_{\eta_0}$ at $p_i'$}
  \Return{$(p_0, \eta_0)$}\;
  \caption{PreliminaryStage}
  \label{algo:IPM pre stage}
\end{algorithm}

\begin{defn}[Symmetry]\label{def:symmetry}
  Let $D \subseteq E$ be a compact convex subset, and let $\vec{p} \in \interior(D)$.
  The \emph{symmetry} of $D$ with respect to $\vec{p}$ is defined by
  \[
    \sym{p} = \max \, \{ a \geq 0 \;:\; \vec{p} + a (\vec{p} - D) \subseteq D \}.
  \]
\end{defn}

\noindent
If $L$ is an affine line through $\vec{p}$, then $L \cap D$ consists of two chords from $\vec{p}$ to the boundary of $D$; the symmetry parameter $\sym{\vec{p}}$ is the smallest possible ratio of the lengths of the smallest and longest chord.
Therefore, the symmetry is always at most $1$, and from this description, it is also clear that one can bound the symmetry by providing a ball centered at $\vec{p}$ contained in the interior of $D$, and another ball centered at $\vec{p}$ containing all of $D$; see \cref{lem:symmetrylemma}.

\begin{thm}[Preliminary stage, {\cite[(2.19)]{renegar-interior-point}}]\label{thm:iteration complexity of IPM preliminary stage}
  Let $\Psi\colon\interior(D) \to \R$ be a strongly non-degenerate self-concordant barrier functional for $D$ with complexity parameter $\nu \geq 1$, let $\vec{p}_0' \in \interior(D)$ be a starting point, and let $c \in E$ be the objective.
  Then~\cref{algo:IPM pre stage} with this choice of starting point~$p'_0$ outputs a vector $\vec{p}_0 \in \interior(D)$ and $\eta_0 > 0$ satisfying \cref{eq:starting condition}, i.e.,
  \begin{align*}
    \norm{H(p_0)^{-1} (\eta_0 c + g(p_0))}_{(p_0)} \leq \frac{1}{9},
  \end{align*}
  after at most
  \begin{align*}
    \frac{\log(18\nu(1+\frac{1}{\sym{\vec{p}_0'}}))}{-\log(1-\frac{1}{8\sqrt{\nu}})} \leq
  8\sqrt{\nu}\log\left( \frac{36\nu}{\sym{\vec{p}_0'}} \right)
  \end{align*}
  iterations. 
  Moreover, we have the lower bound $\eta_0 \geq \frac 1 {12 (V - \val)}$, where $V = \max_{\vec{p} \in D} \, \ip{\vec{c}}{\vec{p}}$.
\end{thm}
\noindent
Together, \cref{thm:iteration complexity of IPM main stage,thm:iteration complexity of IPM preliminary stage} can be summarized as follows:

\begin{thm}[Theorem 2.4.1 in~\cite{renegar-interior-point}]
  \label{thm:renegar algorithm}
  Let $D \subseteq E$ be a compact convex subset with non-empty interior.
  Assume $\Psi\colon \interior(D) \to \R$ is a strongly non-degenerate self-concordant barrier functional for $D$ with complexity parameter $\nu \geq 1$.
  Furthermore, for $\vec{c}\in E$, define $\val=\min_{\vec{p} \in D}\,\ip{\vec{c}}{\vec{p}}$ and $V=\max_{\vec{p} \in D}\,\ip{\vec{c}}{\vec{p}}$.
  Finally, let $0 < \delta < V-\val$ be the desired precision and let $p'_0 \in \interior(D)$ be a starting point for the preliminary stage.
  Then, \cref{algo:IPM pre stage} outputs a point $\vec{p}_0 \in \interior(D)$ and a parameter $\eta_0 \geq \frac1{12(V-\val)}$ satisfying the hypotheses of $\cref{thm:iteration complexity of IPM main stage}$.
  \Cref{algo:IPM main stage} with inputs $\vec{p}_0$, $\eta_0$ and $T \geq 10 \sqrt{\nu} \log(\frac65 \frac{\nu}{\eta_0 \delta})$ outputs a point $\vec{p}_T$ satisfying
  \begin{align*}
    \ip c {p_T} - \val \leq \delta.
  \end{align*}
  The total number of iterations is upper bounded by
  \begin{align*}
    18 \sqrt{\nu}\log\left( \frac{36\nu}{\sym{\vec{p}_0'}} \frac{V-\val}\delta \right)
  \end{align*}
  and each iteration involves computing the gradient and Hessian of the self-concordant barrier $\Psi$ and basic matrix arithmetic.
\end{thm}

\subsection{Barrier method for unconstrained geometric programming}\label{subsec:barrier}
We now return to our concrete situation, and show how to apply the general framework to the unconstrained GP problem.
Fix $\Omega = \{\omega_1, \dotsc, \omega_k\} \subseteq \R^n$, $\vec{q} \in \R_{++}^k$, and a shift~$\theta \in \conv \Omega$.
Following the general strategy outlined above, we first relate the geometric program to the minimization of a linear function over a compact convex domain.
For $R>0$, define
\begin{equation}\label{eq:D_theta R}
\begin{aligned}
  D_{\theta,R} = \Bigl\{ (x,z,t) \in W \times \R^k \times \R \;:\; &\sum_{i=1}^k z_i \leq 1, \;\; q_i e^{\ip {\omega_i-\theta} x} \leq z_i e^t \;\; \forall i\in[k], \\
&t \leq \log(5k\norm q_1), \;\; \norm x_2 \leq R \Bigr\}.
\end{aligned}
\end{equation}
Here we recall that $W$ is the span of the vectors~$\omega_i - \theta$ or, equivalently, the direction vector space of $\affspan\Omega$.
Note $(x,z,t)\in D_{\theta, R}$ implies $z_i>0$ for all $i\in[k]$.
The convexity of the domain~$D_{\theta,R}$ follows from the convexity of the exponential map and of the $\ell_2$-norm ball.
Let $c=(0,\dots,0; 0,\dots,0;1)$.
Then, any point $p=(x,z,t) \in D_{\theta,R}$ consists of a vector $x$ with norm $\norm x_2 \leq R$ and value
\begin{align}\label{eq:gp vs sum z}
  F_\theta(x) = \log \sum_{i=1}^k q_i e^{\ip {\omega_i-\theta} x} \leq \log \sum_{i=1}^k z_i e^t \leq \ip c p = t.
\end{align}
In fact, if $(x,z,t)$ is a $\delta$-approximate minimizer of~\eqref{eq:convex prog} with domain $D_{\theta,R}$ and the above objective, then $x$ is a $\delta$-approximate minimizer of $F_\theta(x)$, restricted to vectors of norm $\norm x_2 \leq R$.
This is implied by the following lemma.

\begin{lem}[Value]
For any $R>0$, we have
\begin{align}
\label{eq:val vs gp}
  \val :=
  \min_{(x,z,t) \in D_{\theta,R}} t &=
  \min_{\norm x_2 \leq R} F_\theta(x) \\
\label{eq:V}
  V :=
  \max_{(x,z,t)\in D_{\theta,R}} t &=
  \log(5k \norm q_1).
\end{align}
As a consequence,
\begin{align}\label{eq:V - val bound}
  \log(5k)
  \leq V - \val \leq
  \log(5k\beta)
\end{align}
\end{lem}
\begin{proof}
For the first claim, note that \cref{eq:gp vs sum z} implies that
\begin{align*}
  \val \geq \min_{\norm x_2 \leq R} F_\theta(x).
\end{align*}
Now consider a minimizer $x$ of the right-hand side, which we by \cref{eq:f stab} can take in $W$.
Then, $t := F_\theta(x)$ is such that
\begin{align*}
  t \leq F_\theta(0) = \log \norm q_1 \leq \log (5 k \norm q_1),
\end{align*}
and if we set $z_i := q_i e^{\ip {\omega_i-\theta} x - t}$ then
\begin{align*}
  \sum_{i=1}^k z_i = \sum_{i=1}^k q_i e^{\ip {\omega_i-\theta} x} e^{-t} = e^{F_\theta(x)} e^{-F_\theta(x)} = 1.
\end{align*}
Thus we find that $(x,z,t) \in D_{\theta,R}$, with $t = F_\theta(x)$, and \cref{eq:val vs gp} follows.

To see that \cref{eq:V} holds, note that the upper bound $V \leq \log(5k\norm q_1)$ follows directly by the constraint on the $t$ variable, and is achieved for the point
\begin{align*}
  p = \bigl( 0,\dots,0; \frac1k, \dots, \frac1k; \log(5k \norm q_1) \bigr).
\end{align*}
Now \cref{eq:V - val bound} follows from the preceding, along with the estimate
\begin{equation*}
  \log \min_{i\in[k]} q_i \leq F_\theta^* \leq \val \leq F_\theta(0) = \log \norm q_1,
\end{equation*}
where the lower bound on $F_\theta^*$ is from \cref{eq:lower bound optimal value}.
\end{proof}

The key to applying interior-point methods to unconstrained geometric programming is the following result, which gives an explicit barrier functional for~$D_{\theta,R}$.
It is well-known to experts and follows from standard barrier functionals and combination rules.

\begin{prop}[Barrier]\label{prop:domainbarrier}
The compact domain $D_{\theta,R} \subseteq W \times \R^k \times \R$ has non-empty interior.
Moreover, it admits the self-concordant barrier functional
\begin{align*}
  \Psi_{\theta,R}(x,z,t)
  = &- \sum_{i=1}^k \log z_i
      - \sum_{i=1}^k \log \left( \log z_i - \ip{\omega_i - \theta}x + t - \log q_i \right) \\
    &- \log \bigl( \log(5k\norm q_1) - t \bigr)
      - \log \bigl( 1 - \textstyle\sum_{i=1}^k z_i \bigr)
      - \log \bigl( R^2 - \norm x_2^2 \bigr),
\end{align*}
with complexity parameter~$\nu=2k+3$.
\end{prop}
\begin{proof}
  It is clear that $D_{\theta,R}$ has non-empty interior (for example, \cref{eq:starting point} below gives a point in the interior).
  We now derive the barrier functional.
  It is well-known that the epigraph of the exponential, given by
  \begin{align*}
    \{(y,z) \in \R \times \R \;:\; e^y \leq z \},
  \end{align*}
  admits the self-concordant barrier functional $(y,z) \mapsto - \log z - \log(\log z - y)$, with complexity parameter $2$; see \cite[Prop.~5.3.3]{nesterov-nemirovskii}.
  Recall also that the logarithmic barrier functional $\tau \mapsto -\log \tau$ for the half line $\R_+ \subseteq \R$ has complexity parameter~$1$.
  Then the closed convex set
  \begin{align}\label{eq:unconstraineddomain}
    \bigl\{(y,z,\tau) \in \R^k \times \R^k \times \R \;:\; e^{y_i} \leq z_i \text{ for all } i\in[k], \; \tau \geq 0 \bigr\}
  \end{align}
  is simply the product of $k$ copies of the epigraph of the exponential and the half line, so a barrier functional $\Psi'$ is given by the sum of the barrier functionals for each term in the product (cf.~\cite[Prop.~2.3.1 (iii)]{nesterov-nemirovskii}), i.e.,
  \begin{align*}
    \Psi'(y, z, \tau) = - \sum_{i=1}^k \log z_i - \sum_{i=1}^k \log(\log z_i - y_i) - \log \tau.
  \end{align*}
  The complexity parameter is then at most the sum of the individual complexity parameters, i.e., $2k + 1$.
  Next, note that
  \begin{align*}
    \bigl\{(x,z,t) \in W \times \R^k \times \R \;:\; q_i e^{\ip {\omega_i-\theta}x} \leq z_i e^t \text{ for all } i\in[k], \; t \leq \log(5k \norm q_1) \bigr\}
  \end{align*}
  is the preimage of~\cref{eq:unconstraineddomain} under the injective affine transformation
  \begin{align*}
    &A \colon W \times \R^k \times \R \to \R^k \times \R^k \times \R, \quad (x,z,t) \mapsto \bigl(\ip {\omega_1-\theta}x - t + \log q_1, \dots \\
    &\qquad \dots, \ip {\omega_k-\theta}x - t + \log q_k; \, z_1, \dots, z_k; \, \log(5k\norm q_1) - t \bigr),
  \end{align*}
  hence by \cite[Prop.~2.3.1 (i)]{nesterov-nemirovskii} admits the self-concordant barrier functional
  \begin{align*}
    (\Psi' \circ A)(x,z,t)
  &= - \sum_{i=1}^k \log z_i - \sum_{i=1}^k \log\left(\log z_i - \ip {\omega_i-\theta}x + t - \log q_i \right) \\
  &- \log \bigl( \log(5k\norm q_1) - t \bigr)
  \end{align*}
  with the same complexity parameter $2k+1$.
  Finally, we may incorporate the linear constraint $\sum_{i=1}^k z_i \leq 1$ by adding the logarithmic barrier~$-\log(1 - \sum_{i=1}^k z_i)$, and the $\ell_2$-norm constraint $\norm x_2 \leq R$ by adding the barrier~$-\log(R^2 - \norm x_2^2$); see \cite[Prop.~2.3.1 (ii)]{renegar-interior-point}.
  This increases the complexity parameter to $2k+3$ and results in the desired domain and barrier.
\end{proof}

Next we need to bound the symmetry of a suitable starting point.
For this, we will use the following lemma.

\begin{lem}\label{lem:symmetrylemma}
  Let $D \subseteq E$ be a closed convex subset, and let $p\in D$.
  Suppose that $r < R$ are two radii such that
  $B(p, r) \subseteq D \subseteq B(p,R)$,
  where the closed balls are taken with respect to an \emph{arbitrary} norm on~$E$.
  Then, $D$ is bounded, $p\in\interior(D)$,~and
  \begin{align*}
    \sym p \geq \frac r R.
  \end{align*}
\end{lem}
\begin{proof}
  Clearly $D$ is bounded and contains~$p$ in its interior.
  For the symmetry claim, note that for any $u\in D$, we have $u \in B(p, R)$, so $p - u \in B(0, R)$, and hence
  \begin{align*}
    p + \frac rR (p - u) \in B(p, r) \subseteq D.
  \end{align*}
  This shows that
  \begin{align*}
    p + \frac rR (p - D) \subseteq D,
  \end{align*}
  which implies the desired lower bound on the symmetry (cf.~\cref{def:symmetry}).
\end{proof}

\noindent
One important aspect of \cref{lem:symmetrylemma} is the freedom in choosing a norm; in particular, we do not assume that the norm comes from the inner product on~$E$.
We will now use this freedom to bound the symmetry of the following starting point:
\begin{align}\label{eq:starting point}
  p'_0 = \bigl( 0,\dots,0; \frac1{2k}, \dots, \frac1{2k}; \log(4k\norm q_1) \bigr),
\end{align}
which is clearly contained in the interior of~$D_{\theta,R}$.

\begin{prop}[Symmetry bound]\label{prop:general symmetry bound}
For any $R>0$, we can bound the symmetry of~$D_{\theta,R}$ with respect to the point~$p'_0$ by
\begin{align*}
    \frac1{\sym{p'_0}}
  \leq 10\max\bigl( R_\theta R, k, \log(4k\beta) \bigr),
\end{align*}
where we recall that $R_\theta = \max_{i\in[k]} \norm{\omega_i - \theta}_2$ as defined in \cref{eq:R_theta}.
\end{prop}
\begin{proof}
  We wish to apply \cref{lem:symmetrylemma} using the following norm on $W \times \R^k \times \R$:
  \begin{align*}
    \tnorm{(x,z,t)} := \max \left\{ \frac{\norm x_2}{R}, \frac23 \norm z_\infty, \frac{\abs t}{\log(4k\beta)} \right\}.
  \end{align*}
  We first show that
  \begin{align}\label{eq:large ball}
    \tnorm{(x,z,t) - p'_0} \leq 1
  \end{align}
  for any $(x,z,t) \in D_{\theta,R}$.
  By definition $\norm x_2 \leq R$.
  Moreover, $z_i > 0$, so
  \begin{align*}
    \frac23 \abs{z_i - \frac1{2k}}
  \leq \frac23 \left( z_i + \frac1{2k} \right)
  \leq \frac23 \left( \sum_{i=1}^k z_i + \frac1{2k} \right)
  \leq \frac23 \left( 1 + \frac1{2k} \right)
  \leq 1.
  \end{align*}
  Moreover, by~\cref{eq:lower bound optimal value,eq:val vs gp,eq:V} it holds that
  \begin{align*}
    \log \min_{i\in[k]} q_i \leq F_\theta^* \leq \val \leq t \leq V = \log (5k \norm q_1),
  \end{align*}
  hence
  \begin{align*}
  \frac{\abs{t - \log(4k\norm q_1)}}{\log(4k\beta)}
  &\leq \frac{\max \{ \log (5k\norm q_1) - \log(4k\norm q_1), \, \log(4k\norm q_1) - \log \min_{i\in[k]} q_i \} }{\log(4k\beta)} \\
  &= \frac{\max \{ \log \frac54, \log(4k\beta) \}}{\log(4k\beta)} \leq 1.
  \end{align*}
  Thus we have proved \cref{eq:large ball}.

  We now show that $D_{\theta,R}$ contains any point $(x,z,t) \in W \times \R^k \times \R$ in the ball
  \begin{align}\label{eq:small ball}
    \tnorm{(x,z,t) - p'_0}
  \leq \frac1{10\max(R_\theta R, k, \log(4k\beta))}
  < \frac1{10}.
  \end{align}
  The latter implies that $\norm x_2 \leq \frac R {10} \leq R$, so $x$ certainly satisfies the norm bound.
  Moreover, \cref{eq:small ball} ensures that
  $\norm x_2 \leq \frac1{10 R_\theta},$
  and so
  \begin{align}\label{eq:q upper}
    q_i e^{\ip {\omega_i-\theta} x}
  \leq q_i e^{R_\theta \norm x_2}
  \leq q_i e^{1/10}
  \leq e^{1/10} \norm q_1
  \end{align}
  for all $i\in [k]$.
  Next, we also have $\frac23\abs{z_i - \frac1{2k}} \leq \frac1{10k}$, hence
  \begin{align}\label{eq:z range}
    \frac{7}{20k} \leq z_i \leq \frac{13}{20k},
  \end{align}
  which implies that
  \begin{align*}
    \sum_{i=1}^k z_i \leq \frac{13}{20} \leq 1.
  \end{align*}
  Finally, note that \cref{eq:small ball} entails
  \begin{align*}
    \frac{\abs{t - \log(4k\norm q_1)}}{\log(4k\beta)} \leq \frac1{10\log(4k\beta)},
  \end{align*}
  hence $\abs{t - \log(4k\norm q_1)} \leq \frac1{10}$, which implies that
  \begin{align*}
  \log(e^{-1/10}4 k\norm q_1)
  \leq t \leq
  \log(e^{1/10}4 k\norm q_1)
  \leq \log(5 k\norm q_1).
  \end{align*}
  On the one hand, this yields the desired upper bound on~$t$.
  On the other hand, the lower bound, along with the lower bound in \cref{eq:z range}, implies that
  \begin{align}\label{eq:zet lower}
    z_i e^t
  \geq \frac{7}{20k} e^{-1/10}4 k\norm q_1
  = \frac75 e^{-1/10} \norm q_1
  \geq e^{1/10} \norm q_1.
  \end{align}
  Together, \cref{eq:q upper,eq:zet lower} show that $q_i e^{\ip {\omega_i-\theta} x} \leq z_i e^t$, as desired.
  Thus we have proved that $(x,z,t) \in D_{\theta,R}$ for any point in the ball~\eqref{eq:small ball}.
  Now \cref{lem:symmetrylemma} implies the bound on the symmetry.
\end{proof}

In the remainder we consider two different situations.
For general instances, we choose $R$ according to a given lower bound on the facet gap, using \cref{thm:facet gap radius bound}.
In the well-conditioned case, where $\theta$ is contained in the relative interior of the Newton polytope, we see that the upper bound on the~$z_i$ variables already leads to a bounded domain;
this allows us to obtain an algorithm that is independent of any explicit radius bound.

\subsection{General instances}
Suppose the facet gap of the instance is lower bounded by some $\varphi_0>0$.
Then, \cref{thm:facet gap radius bound,eq:val vs gp} show that for $\delta<4\beta$ and
\begin{align*}
  R = \frac n {\varphi_0} \log\left( \frac{2 \beta}{\delta/2} \right)
\end{align*}
the value of the convex program~\eqref{eq:convex prog} with objective $c=(0,\dots,0;0,\dots,0;1)$ and domain~$D_{\theta,R}$ is at most
\begin{align*}
  \val = \min_{\norm x_2 \leq R} F_\theta(x) \leq F_\theta^* + \frac\delta2.
\end{align*}
Therefore, in order to obtain a $\delta$-approximate minimizer for the geometric program, it suffices to find a $\delta/2$-approximate minimizer of the convex program~\eqref{eq:convex prog}.
The latter is achieved by \cref{algo:ipm general}, which is an interior-point algorithm for the self-concordant barrier functional $\Psi_{\theta,R}$ derived above.
Its iteration complexity is bounded by the following theorem.

\begin{algorithm}[t]
  \Input{exponents $\omega_1, \dotsc, \omega_k \in \R^n$, coefficients $\vec q \in \R^k_{++}$, shift $\theta \in \conv\Omega$, precision $0 < \delta < 1$, lower bound $\varphi_0$ on facet gap}
  \Domain{$D_{\theta,R} = \{ (\vec x, \vec z, t) \in W \times \R^k \times \R \;:\;q_i e^{\ip {\omega_i-\theta} x} \leq z_i e^t \;\; \forall i \in [k]$,
    $\sum_{i=1}^k z_i \leq 1$, $\norm x_2 \leq R$ and $t \leq \log(5k\norm{q}_1) \}$, where
    $W = \linspan \{\omega_1 - \theta, \dotsc, \omega_k - \theta \}$, and
    $R = \frac n {\varphi_0} \log( \frac{2 \beta}{\delta/2} )$}
  \Barrier{$\Psi_{\theta,R}(x,z,t) = - \log\bigl(R^2 - \norm{\vec x}_2^2\bigr) - \log \bigl( 1 - \sum_{i=1}^k z_i \bigr) - \log\bigl(\log(5 k \norm{q}_1) - t\bigr) + \sum_{i=1}^k - \log z_i - \log \bigl( \log z_i - \ip{\omega_i - \theta}x - \log q_i + t \bigr)$}
  \ComplexityParameter{$\nu = 2k + 3$}
  \BlankLine
  \Indp
  $\vec{p}'_0 \leftarrow \bigl( 0,\dots,0; \frac1{2k}, \dots, \frac1{2k}; \log(4k\norm q_1) \bigr)$\;
  $\vec{c} \leftarrow (0,\dots,0; 0,\dots,0; 1)$\;
  $(\vec{p}_0, \eta_0) \leftarrow$ \PreliminaryStage{$\vec{p}_0'$, $c$}\;
  $(x,z,t) \leftarrow$ \MainStage{$\vec{p}_0$, $\eta_0$, $T = 10 \sqrt{\nu} \log(\frac{6}{5} \frac{\nu}{\eta_0 \delta/2})$, $\vec{c}$}\;
  \Return $\vec t$
  \caption{IPM for unconstrained GP: general case}
  \label{algo:ipm general}
\end{algorithm}

\begin{ugpthmgeneral}[restated]
\detailedboundstrue\ugpthmgeneralcontent
\end{ugpthmgeneral}
\begin{proof}
  The result follows from applying \cref{thm:renegar algorithm} to find a $\frac\delta2$-approximate minimizer, with the closed convex domain $D_{\theta,R}$, the self-concordant barrier functional~$\Psi_{\theta,R}$ from \cref{prop:domainbarrier} with complexity parameter~$\nu=2k+3$, the symmetry bound given by \cref{prop:general symmetry bound}, and the starting point~\eqref{eq:starting point}, along with the estimate $\log(5k) \leq V - \val \leq \log(5k\beta)$ from \cref{eq:V - val bound} and the bound $R_\theta = \max_i \norm{\omega_i - \theta}_2 \leq N = \max_{i\neq j} \norm{\omega_i - \omega_j}_2$ which holds for any $\theta\in\conv\Omega$.
\end{proof}

If all the inputs for \cref{algo:ipm general} are rational and encoded in binary, then $\norm{\vec{q}}_1$, $\beta$, and $\varphi_0$ are at most exponentially large in the encoding length.
Since the iteration complexity depends logarithmically (or even doubly logarithmically) on these quantities, the resulting iteration complexity is at most \emph{polynomial} in the encoding length of the input.
See \cref{section:total unimodularity} for details.

\subsection{Well-conditioned instances}\label{subsec:well cond}
Now assume the instance is well-conditioned, so $\theta$ is contained in the relative interior of the Newton polytope.
Here, we consider
\begin{align*}
  D_\theta = \Bigl\{ (x,z,t) \in W \times \R^k \times \R \;:\; &\sum_{i=1}^k z_i \leq 1, \;\; q_i e^{\ip {\omega_i-\theta} x} \leq z_i e^t \;\; \forall i\in[k], \\
&t \leq \log(5k\norm q_1) \Bigr\},
\end{align*}
which looks just like $D_{\theta,R}$ except that we omitted the norm bound on $x$.
We claim that the two domains coincide for any
\begin{align}\label{eq:R bound wc}
  R \geq \frac{\log(5k\beta)}{r_\theta}.
\end{align}
Indeed, if there were some $(x,z,t) \in D_\theta$ with $\norm x_2 > R$ then \cref{lem:wc large x} would show that $q_{i_0} e^{\ip{\omega_{i_0} - \theta} x} > 5k \norm q_1$ for some $i_0\in[k]$.
This is a contradiction, since for any $(x,z,t)\in D_\theta$ we have
\begin{align*}
  q_{i_0} e^{\ip{\omega_{i_0} - \theta} x}
\leq \sum_{i=1}^k q_i e^{\ip {\omega_i-\theta} x}
\leq \sum_{i=1}^k z_i e^t \leq e^t \leq 5 k \norm q_1.
\end{align*}
Thus we see that, indeed, $D_\theta = D_{\theta,R}$ for any~$R$ as in \cref{eq:R bound wc}.

As a consequence, the value of the convex program for the domain $D_\theta$ is exactly equal to $F_\theta^*$, as follows from \cref{eq:val vs gp}.
Moreover, the domain $D_\theta$ is bounded and satisfies the symmetry bound given in \cref{prop:general symmetry bound}, namely,
\begin{align}\label{eq:wc symmetry bound}
  \frac1{\sym{p'_0}} \leq
10\max\bigl( \log(5k\beta) \frac{R_\theta}{r_\theta}, k, \log(4k\beta) \bigr)
= 10\max\bigl( \log(5k\beta) \frac{R_\theta}{r_\theta}, k \bigr)
\end{align}
Moreover, since $D_\theta$ no longer depends explicitly on the radius bound, we can use the self-concordant barrier functional
\begin{equation}\label{eq:wc barrier}
\begin{aligned}
  \Psi_{\theta}(x,z,t)
  = &- \sum_{i=1}^k \log z_i
  - \sum_{i=1}^k \log \left( \log z_i - \ip{\omega_i - \theta}x + t - \log q_i \right) \\
  &- \log \bigl( \log(5k\norm q_1) - t \bigr)
  - \log \bigl( 1 - \textstyle\sum_{i=1}^k z_i \bigr)
\end{aligned}
\end{equation}
with complexity parameter~$\nu=2k+2$.
Using this modification we readily obtain an interior-point algorithm for well-conditioned instances.
Importantly, this algorithm does \emph{not} explicitly depend on~$r_\theta$ or any other condition measure.
By contrast, \cref{algo:ipm general} required as input a lower bound on the facet gap.
The algorithm is stated in \cref{algo:ipm well-conditioned}, and the following theorem gives a precise iteration bound.

\begin{algorithm}[t]
  \Input{exponents $\omega_1, \dotsc, \omega_k \in \R^n$, coefficient vector $\vec q \in \R^k_{++}$, shift~$\theta \in \relint\conv\Omega$, precision $0 < \delta < 1$}
  \Domain{$D_{\theta} = \{ (\vec x, \vec z, t) \in W \times \R^k \times \R \;:\;q_i e^{\ip {\omega_i-\theta} x} \leq z_i e^t \;\; \forall i \in [k]$,
    $\sum_{i=1}^k z_i \leq 1$ and $t \leq \log(5k\norm{q}_1) \}$, where
    $W = \linspan \{\omega_i - \theta \}$
  }
  \Barrier{$\Psi_{\theta}(x,z,t) = - \log \bigl( 1 - \sum_{i=1}^k z_i \bigr) - \log\bigl(\log(5 k \norm{q}_1) - t\bigr) + \sum_{i=1}^k - \log z_i - \log \bigl( \log z_i - \ip{\omega_i - \theta}x - \log q_i + t \bigr) $}
  \ComplexityParameter{$\nu = 2k+2$}
  \BlankLine
  \Indp
  $\vec{p}'_0 \leftarrow \bigl( 0,\dots,0; \frac1{2k}, \dots, \frac1{2k}; \log(4k\norm q_1) \bigr)$\;
  $\vec{c} \leftarrow (0,\dots,0; 0,\dots,0; 1)$\;
  $(\vec{p}_0, \eta_0) \leftarrow$ \PreliminaryStage{$\vec{p}_0'$, $c$}\;
  $(x,z,t) \leftarrow$ \MainStage{$\vec{p}_0$, $\eta_0$, $T = 10 \sqrt{\nu} \log(\frac{6}{5} \frac{\nu}{\eta_0 \delta})$, $\vec{c}$}\;

  \Return $\vec t$
  \caption{IPM for unconstrained GP: well-conditioned case}
  \label{algo:ipm well-conditioned}
\end{algorithm}

\begin{ugpthmwc}[restated]
\detailedboundstrue\ugpthmwccontent
\end{ugpthmwc}
\begin{proof}
  Apply \cref{thm:renegar algorithm} to find a $\delta$-approximate minimizer, with the closed convex domain $D_\theta$, the self-concordant barrier functional~$\Psi_\theta$ given in \cref{eq:wc barrier} with complexity parameter~$\nu=2k+2$, the symmetry bound given in \cref{eq:wc symmetry bound}, and the starting point~\eqref{eq:starting point}, along with the estimate on $(V - \val)$ from \cref{eq:V - val bound}.
\end{proof}

As in the situation of \cref{thm:algo general}, if all the inputs in \cref{algo:ipm well-conditioned} are rational, then the iteration complexity is again at most polynomial in the encoding length of the inputs.
Again see \cref{section:total unimodularity} for details.

\subsection{Geometric programming and scaling}
\label{subsec:scaling}
In this section, we show that in order to solve the scaling problem with precision $\epsilon > 0$, it suffices to solve the corresponding unconstrained geometric program with some precision $\delta = \delta(\epsilon)$.
The results in this section are well-known (see, e.g., \cite[Lem.~5.3]{straszak-vishnoi-bitcomplexity} or \cite[Cor.~1.18]{bfgoww-noncommutative-optimization}), but stated and proven for completeness.

\begin{lem}[Smoothness]\label{lem:logobjectivesmooth}
  For any $\omega_1, \dotsc, \omega_k, \theta \in \R^n$ and $q\in\R_{++}^k$, the function
  \begin{align*}
    F_\theta \colon \R^n \to \R, \quad F_\theta(x) = \log \sum_{i=1}^k q_i e^{\ip{\omega_i - \theta}x}
  \end{align*}
  is $L$-smooth with $L = R_\theta^2$, where $R_\theta = \max_i \norm{\omega_i - \theta}_2$.
  Recall that this means that its gradient is $L$-Lipschitz or, equivalently, that its Hessian has eigenvalues~$\leq L$.
\end{lem}
\begin{proof}
  The gradient $\nabla F_\theta(x) \in \R^n$ is given by
  \begin{align*}
    \nabla F_\theta(x) = \frac{\sum_{i=1}^k q_i e^{\ip{\omega_i - \theta}x} (\omega_i - \theta)}{\sum_{i=1}^k q_i e^{\ip{\omega_i - \theta}x}}.
  \end{align*}
  Therefore, its Hessian $\nabla^2 F_\theta(x)\colon \R^n \to \R^n$ is given by
  \begin{align*}
    \nabla^2 F_\theta(x)
  = \frac{\sum_{i=1}^k q_i e^{\ip{\omega_i - \theta}x} (\omega_i - \theta)(\omega_i - \theta)^T}{\sum_{i=1}^k q_i e^{\ip{\omega_i - \theta}x}}
  -  (\nabla F_\theta(x))(\nabla F_\theta(x))^T.
  \end{align*}
  Hence we see that the eigenvalues of the Hessian can be upper bounded by the eigenvalues of the left-hand side matrix, since the matrix that is subtracted is positive semidefinite.
  As the left-hand side matrix is a convex combination of the rank-one matrices $(\omega_i - \theta)(\omega_i - \theta)^T$, we can bound its eigenvalues by $R_\theta^2$
\end{proof}

The following proposition then shows that the scaling problem can be solved by solving the corresponding geometric program with sufficient precision.

\begin{prop}[Scaling from optimization]\label{prop:momentmapaccuracy}
  Assume that $\theta\in\conv\Omega$, and let $x\in\R^n$ be such that $F_\theta(x) \leq F_\theta^* + \delta$ for some $\delta>0$.
  Then,
  \begin{align*}
    \frac {\norm{\grad F_\theta(\vec{x})}_2^2} {2 R_\theta^2} \leq \delta.
  \end{align*}
  In particular, to solve the scaling problem with precision~$\epsilon>0$ it suffices to find a solution for the unconstrained GP with accuracy $\delta = \epsilon^2 / (2 R_\theta^2)$.
\end{prop}
\begin{proof}
  A standard argument shows that an $L$-smooth function can always be decreased in controlled way by following a gradient step.
  Namely, if we define $x' = x - \frac1L\nabla F_\theta(x)$ then, using Taylor's expansion to second order and bounding the quadratic contribution using smoothness,
  \begin{align*}
    F_\theta(x') - F_\theta(x)
  \leq - \frac1L \norm{\grad F_\theta(\vec{x})}_2^2 + \frac1{2L} \norm{\grad F_\theta(\vec{x})}_2^2
  = -\frac1{2L} \norm{\grad F_\theta(\vec{x})}_2^2.
  \end{align*}
  As $x$ is a $\delta$-approximate minimizer of $F_\theta$, we must have
  \begin{align*}
    \frac1{2L} \norm{\grad F_\theta(\vec{x})}_2^2 \leq \delta.
  \end{align*}
  The desired bound follows since we have $L=R_\theta^2$ by \cref{lem:logobjectivesmooth}.
\end{proof}

\Cref{cor:algorithm for scaling,cor:algorithm for uniform scaling} follow directly from \cref{thm:algo wc,thm:algo general}, respectively, by using \cref{prop:momentmapaccuracy}.
Interestingly, the above results also hold in the more general `non-commutative' setting discussed in \cref{subsubsec:opti}; see~\cite{bfgoww-noncommutative-optimization} for details.

\section{Bounds on condition measures}\label{section:total unimodularity}
In this section, we give bounds on the condition measures from \cref{section:condition and diameter} for \emph{rational} instances in terms of their binary encoding length.
These bounds show that our interior-point algorithms have polynomial iteration complexity.
We also explain how to obtain tighter estimates under a total unimodularity assumption on the Newton polytope.
Throughout this section, we follow the conventions of~\cite{gls}:
we encode rational numbers and vectors in binary, and write~$\enclen{\cdot}$ for the encoding length.

\subsection{General bounds}\label{subsec:general bounds}
We first give lower bounds on~$r_\theta$ and~$\varphi$, the distance of~$\theta$ to the boundary of the Newton polytope and the facet gap of~$\Omega$, respectively.
All other condition measures can be directly bounded in terms of the input length.
The following bound on~$r_\theta$ implies \cref{eq:upper bound gcn by bit complexity} in the introduction.

\begin{lem}\label{lem:geometricconditionbound}
  Let $\Omega \subseteq \Q^n$ and $\theta\in\Q^n\cap\relint\conv\Omega$.
  Then,
  \begin{align*}
    \log_2 \frac1{r_\theta} \leq 6n^2 \max_{i \in [k]} \, \enclen{\omega_i} + \enclen{\theta} - n.
  \end{align*}
  If $\theta = \vec 0$, the upper bound can be improved to $3n^2 \max_{i \in [k]} \, \enclen{\omega_i} - n$.
\end{lem}
\begin{proof}
  The polytope $\conv \Omega$ has vertex complexity at most $\nu := \max_{i \in [k]} \, \enclen{\omega_i}$, so by \cite[Lem.~6.2.4]{gls}, it has facet complexity at most $\phi := 3 n^2 \nu$.
  This means that the polytope can be defined by inequalities of the form $\ip \cdot a \leq b$ for $a\in \Q^n$, $b\in \Q$ with encoding length $\enclen a + \enclen b \leq \phi$.

  As a consequence, if $F$ is any facet of $\conv \Omega$ then its distance to $\theta$ can be lower bounded as
  \begin{align*}
    d(\theta, F) \geq
    d(\theta, \affspan F) \geq
    \frac{b - \ip\theta a}{\norm a_2}
  \end{align*}
  for certain $a\in\Q^n$, $b\in\Q$ with $\enclen a + \enclen b \leq \phi$.
  Now we have $\norm a_2 \leq 2^{\enclen a - n}$ by \cite[Lem.~1.3.3]{gls}, while $b - \ip\theta a$ is a positive rational number with denominator of absolute value at most $2^{\enclen a + \enclen b + \enclen \theta} \leq 2^{\phi + \enclen\theta}$.
  We conclude that the distance from~$\theta$ to the facet~$F$ is at least
  \begin{align*}
    \frac{b - \ip\theta a}{\norm a_2} \geq
    \frac1{2^{\phi + \enclen\theta} 2^{\enclen a - n}} \geq
    \frac1{2^{2\phi + \enclen\theta - n}} =
    \frac1{2^{6 n^2 \max_i \enclen{\omega_i} + \enclen\theta - n}}.
  \end{align*}
  Since the facet was arbitrary this implies the desired bound.

  If $\theta=0$, then we can instead estimate
  \begin{align*}
    \frac{b - \ip\theta a}{\norm a_2} =
    \frac b{\norm a_2} \geq
    \frac1{2^{\enclen b} 2^{\enclen a - n}} \geq
    \frac1{2^{\phi - n}} \geq
    \frac1{2^{3 n^2 \max_i \enclen{\omega_i} - n}},
  \end{align*}
  which proves the second claim.
\end{proof}

A completely similar argument shows the following lemma, which implies \cref{eq:intro facet gap bit bound} in the introduction.

\begin{lem}\label{lem:facet gap bit bound}
  Let $\Omega \subseteq \Q^n$.
  Then the facet gap $\varphi$ of $\Omega$ satisfies
  \begin{align*}
    \log_2 \frac1\varphi \leq (6n^2 + 1) \max_{i \in [k]} \, \enclen{\omega_i} - n.
  \end{align*}
\end{lem}
\begin{proof}
  The argument is exactly the same as in the proof of \cref{lem:geometricconditionbound}, but with $\theta$ replaced by any $\omega_i$ not on the facet under consideration, noting that the proof in fact established a lower bound on the distance to the \emph{affine span} of the facet.
\end{proof}

Finally, we show that the unary facet complexity of an integral polytope can be similarly bounded in terms of the encoding length.
Via \cref{prop:facet gap ufc}, this also implies a bound on the facet gap, albeit with a worse polynomial scaling in the dimension~$n$.

\begin{lem}\label{lem:ufcbound}
  Let $\Omega \subseteq \Z^n$. Then the unary facet complexity of $\conv \Omega$ satisfies
  \begin{align*}
    \log_2 \ufc(\conv{\Omega}) \leq 3 n^3 \max_{i \in [k]} \, \enclen{\omega_i} - n.
  \end{align*}
\end{lem}
\begin{proof}
  Again by \cite[Lem.~6.2.4]{gls}, the polytope $\conv\Omega$ may be described by inequalities of the form $\ip p a \leq b$ with $a \in \Q^n$, $b \in \Q$ of total encoding length $\enclen a + \enclen b \leq \phi := 3n^2 \max_{i\in[k]} \, \enclen{\omega_i}$.
  Multiplying by the denominators of $a$ gives $a' \in \Z^n$, $b' \in \Q$ such that $a'$ has encoding length at most $n\phi = 3n^3 \max_{i\in[k]} \, \enclen{\omega_i}$.
  Now the desired inequality follows from the bound $\norm{a'}_\infty \leq \norm{a'}_2 \leq 2^{\enclen{a'}-n}$.
\end{proof}

\subsection{Total unimodularity}
We now show how to improve the bounds given above in case the set of exponents $\Omega$ satisfies a total unimodularity hypothesis.
This is the case in many interesting applications, including the geometric programs~\eqref{eq:quiver f} associated with directed graphs, which in particular capture the matrix scaling and matrix balancing problems.

\begin{defn}[Total unimodularity]
  An integer matrix $A \in \Z^{n \times k}$ is called \emph{totally unimodular} if every square submatrix of $A$ has determinant $0$, $1$ or $-1$.
  We say that $\Omega=\{\omega_1,\cdots,\omega_k\} \subseteq \Z^n$ is \emph{totally unimodular} if the associated matrix
  \begin{align*}
    A_\Omega = \bigl[ \omega_1 \big| \cdots \big| \omega_k \bigr]
  \end{align*}
  with columns $\omega_1,\dots,\omega_k$ is totally unimodular.
\end{defn}

If~$\Omega$ is totally unimodular, every $\omega_i$ has entries only in $\{\pm1,0\}$.
Therefore, we can bound the radius of the smallest enclosed ball around any $\theta\in\conv\Omega$, as well as the diameter of the Newton polytope by
\begin{align}\label{eq:R_theta N bounds tu}
  R_\theta = \min_{i\in[k]} \norm{\omega_i - \theta}_2 \leq N = \max_{i\neq j} \norm{\omega_i - \omega_j}_2 \leq 2\sqrt n.
\end{align}
We now show that the inverse distance to the boundary and the inverse facet gap can similarly be upper bounded by a polynomial in~$n$, which is an exponential improvement over the general bounds of \cref{lem:geometricconditionbound,lem:facet gap bit bound}.
The proposition establishes in particular \cref{eq:tu bound on facet gap} in the introduction.

\begin{thm}[Totally unimodular bounds]\label{thm:TUconditionbounds}
  Let $\Omega\subseteq\Z^n$ be totally unimodular.
  Then the unary facet complexity $\ufc(\conv \Omega)$ is at most $n$.
  As a consequence, the facet gap $\varphi$ of $\Omega$ satisfies
  \begin{align*}
    \varphi \geq n^{-3/2}.
  \end{align*}
  Furthermore, if $\theta\in\Q^n\cap\relint\conv\Omega $, then
  \begin{align*}
    r_\theta \geq 2^{-\enclen{\theta}} \, n^{-3/2}.
  \end{align*}
  If $\theta=0$, the latter lower bound can be improved to $n^{-3/2}$.
\end{thm}
\begin{proof}
  Assume first that $\conv \Omega$ is a full-dimensional polytope.
  Then every facet of $\conv \Omega$ is the convex hull of some affinely independent $\vec{v}_1, \dotsc, \vec{v}_n \in \Omega$.
  By Cramer's rule, the affine hyperplane spanned by the facet consists of all $x \in \R^n$ such that
  \begin{align*}
    \det
    \begin{bmatrix}
      1 & 1 & \dotsc & 1 \\
      x_1 & & & \\
      \vdots & \vec{v_1} & \dotsc & \vec{v_n} \\
      x_n & & &
    \end{bmatrix} = 0.
  \end{align*}
  Expanding the determinant along the first column gives the linear equation
  \begin{align}\label{eq:cramer facet}
    \sum_{i=1}^n (-1)^i \det (D_i) x_i = -\det (D_0)
  \end{align}
  where $D_i$ is obtained by deleting the $(i+1)$-th row from the matrix
  \begin{align*}
    D =
    \begin{bmatrix}
      1 & \dotsc & 1 \\
      v_{1,1}&\dotsc& v_{n,1}\\
      \vdots & &\vdots\\
      v_{1,n} & \dotsc & v_{n,1}\\
    \end{bmatrix}.
  \end{align*}
  For $i,j\in[n]$, let $D_i^j$ be obtained by deleting the first row and the $j$-th column from $D_i$; then expanding the determinant $\det(D_i)$ along the first row gives
  \begin{align}\label{eq:ufc coeffs}
    \abs{\det(D_i)} \leq \sum_{j=1}^n \abs{\det(D_i^j)} \leq n
  \end{align}
  since $D_i^j$ is a submatrix of $A_\Omega$ and hence submodular.
  By replacing the equality in \cref{eq:cramer facet} by an inequality and varying over all facets, we obtain a complete set of defining inequalities for the polytope.
  This shows that the unary facet complexity $\ufc(\conv \Omega)$ is at most~$n$.
  The lower bound on the facet gap now follows at once using \cref{prop:facet gap ufc}.

  We now consider an arbitrary $\theta\in\Q^n$ in the interior of $\conv \Omega$ and bound its distance to the boundary.
  Set $\vec{a} = [\det(D_1), \dots, \det(D_n)]^T$ and $b = -\det(D_0)$, so that the hyperplane defined by \cref{eq:cramer facet} reads $\ip x a = b$.
  The distance from $\theta$ to the facet is then lower bounded by
  \begin{align*}
    \frac{\abs{b - \ip \theta a}}{\norm{a}_2} \geq \frac{\abs{b - \ip \theta a}}{n^{3/2}},
  \end{align*}
  where we used that $\norm a_2 \leq \sqrt n \norm a_\infty \leq n^{3/2}$ by \cref{eq:ufc coeffs}.
  Note that $\ip \theta a \neq b$, since $\theta$ is not contained in the hyperplane.
  Moreover, $a\in\Z^n$ and $b\in\Z$.
  Therefore, if $\theta = 0$ then $\abs{b - \ip \theta a} = \abs b$ is an integer (in fact, equal to $1$ by total unimodularity), while in general it is a rational number with denominator at most $2^{\enclen\theta}$.
  In either case we obtain the desired lower bound on $r_\theta$.

  Finally, suppose that $\conv \Omega$ has dimension~$r < n$.
  Then there exists a set of vectors $U = \{\vec{u}_1, \dotsc, \vec{u}_{n-r}\}$ in $\{\vec{0}, \vec{e}_1, \dotsc, \vec{e}_n\}$ such that $\conv (\Omega \cup U)$ has dimension~$n$.
  Moreover, $\Omega \cup U$ is still totally unimodular.
  Hence by the previous part of the proof, $\conv (\Omega \cup U)$ has unary facet complexity at most $n$.
  Every facet of $\conv \Omega$ is now the intersection of some facet of $\conv (\Omega \cup U)$ with the affine span of $\Omega$, so the unary facet complexity of $\conv \Omega$ is also at most $n$.
  Furthermore, since the distance from~$\theta$ to any facet of $\conv \Omega$ is at least as large as the distance from $\theta$ to any facet of $\conv (\Omega \cup U)$ not containing $\theta$, we also inherit the lower bound on~$r_\theta$.
\end{proof}

The lower bound $r_{\vec{0}}^{-1} \geq n^{-3/2}$ when $\theta = \vec{0}$ already appears in \cite[Cor.~6.11]{bfgoww-noncommutative-optimization} as a lower bound on the \textit{weight margin} $\gamma(\pi)$ of a representation $\pi: \T(n) \to \GL(V)$ whose weights are exactly~$\Omega$.
The proof given there is similar to the one we give (as well as to the proof of \cite[Lem.~6.2.4]{gls}, which is also a key ingredient for \cref{lem:geometricconditionbound}):
both use Cramer's rule to express equations for facets of $\conv \Omega$ in terms of subdeterminants of the matrix~$A_\Omega$, which are bounded by the total unimodularity.

The following corollary specializes \cref{thm:algo general,cor:algorithm for uniform scaling} to the totally unimodular case, using \cref{eq:R_theta N bounds tu} and the lower bound on the facet gap from \cref{thm:TUconditionbounds}.

\begin{cor}\label{cor:tu uniform}
  There is an interior-point algorithm (\cref{algo:ipm general}) that, given as input an instance of the unconstrained GP problem with shift and totally unimodular~$\Omega\subseteq\Z^n$, returns $\vec{x}_\delta \in \R^n$ such that $F_\theta(\vec{x}_\delta) \leq F_\theta^* + \delta$ within
  \begin{align*}
    O\left( \sqrt{k} \log\left( k n \frac1\delta \log\left( \frac{k\beta}{\delta} \right) \right) \right)
    = \tilde O\left( \sqrt{k} \log\left( \frac1\delta \right) \right)
  \end{align*}
  iterations.
  Similarly, given an instance of the scaling problem with totally unimodular~$\Omega\subseteq\Z^n$, the same algorithm returns~$\vec{x}_{\epsilon} \in \R^n$ such that $\norm{\grad F_\theta(\vec{x}_\epsilon)}_2 \leq \epsilon$~within
  \begin{align*}
    O\left( \sqrt{k} \log\left( k n \frac1{\epsilon} \log\left( \frac{kn\beta}{\epsilon} \right) \right) \right)
    = \tilde O\left( \sqrt{k} \log\left( \frac1{\epsilon} \right) \right)
  \end{align*}
  iterations.
  Here, the notation $\tilde{O}$ hides poly(input) terms inside the logarithm.
\end{cor}

In particular, this theorem applies to matrix scaling, as discussed below \cref{eq:quiver f}, by using the following geometric program which is totally unimodular:
\begin{align*}
  F_\theta(x,y)
= \log \sum_{i,j} q_{ij} e^{x_i -  y_j - \ip r x + \ip c y}
= \log \biggl( \sum_{i,j} q_{ij} e^{x_i -  y_j} \biggr) - \ip r x + \ip c y.
\end{align*}
Here, slightly stronger bounds can be obtained:
the diameter of the Newton polytope is $N=2$ and the facet gap satisfies $\varphi \geq n^{-1/2}$, since the unary facet complexity of the Newton polytope is in fact equal to~$1$~\cite{straszak-vishnoi-bitcomplexity}.
However, this does not impact the iteration count up to logarithmic factors in the encoding length.

For matrix scaling, the state of the art for general matrices is a recent interior-point method given in~\cite[Thm.~6.1]{cmtv-matrix-scaling}, which obtains an iteration complexity of
\begin{align}\label{eq:cohen}
  \tilde{O}\left(\sqrt{k} \log\left( \frac{\norm{\vec q}_1}{\epsilon}\right)\right)
\end{align}
to find an $(r,c)$-scaling of a nonnegative matrix.
They use an objective that is slightly different from our $F_\theta$, namely
\begin{align*}
  \tilde f_\theta(x,y) = \sum_{i,j} q_{ij} e^{x_i -  y_j} - \ip r x + \ip c y,
\end{align*}
that is, the `shift' is done additively instead of in the exponent.
We see that the iteration complexity in \cref{cor:tu uniform} slightly improves over \cref{eq:cohen}.

\bibliography{references}
\bibliographystyle{amsplain}

\end{document}